\documentclass[12pt]{article}
\usepackage{tipa}
\usepackage[all,import]{xy}
\usepackage{graphicx,color}
\usepackage{amsmath, amsthm}
\usepackage{amsbsy}
\usepackage{amssymb}
\usepackage{cases}
\usepackage{enumerate}
\usepackage{bm}
\usepackage[colorlinks,linkcolor=black,anchorcolor= black,citecolor= black,urlcolor=black]{hyperref}
\usepackage{multirow}

\usepackage[super, comma, sort&compress]{natbib}

\usepackage{times}
\usepackage{mathptmx}

\newtheorem{proposition}{Proposition}

\newtheorem{lemma}{Lemma}

\def\ind{\begin{picture}(9,8)
         \put(0,0){\line(1,0){9}}
         \put(3,0){\line(0,1){8}}
         \put(6,0){\line(0,1){8}}
         \end{picture}
        }

\def\RD{\textsc{RD}}
\def\RR{\textsc{RR}}
\def\OR{\textsc{OR}}
\def\CRR{\textsc{CRR}}
\def\GRR{\textsc{RR}}
\def\GOR{\textsc{OR}}
\def\SRR{\textsc{RR}^\true}
\def\CRD{\textsc{CRD}}
\def\SRD{\textsc{RD}^\true}
\def\GRD{\textsc{RD}}
\def\pr{\text{P}}
\def\true{\text{true}}    
\def\V{\mathbf{V}}
\def\BF{\textsc{BF}}
\def\HR{\textsc{HR}}
\def\SHR{\textsc{HR}^\true}
\def\CHR{\textsc{CHR}}
\def\E{\mathbb{E}}
\def\MR{\textsc{MR}}
\def\SMR{\textsc{MR}^{\true}}
\def\CMR{\textsc{CMR}}
\def\ACE{\textsc{ACE}^\true}
\def\obs{\text{obs}}

\makeatletter
\newcommand*{\rom}[1]{\expandafter\@slowromancap\romannumeral #1@}
\makeatother

\begin{document}

\setlength{\baselineskip}{2\baselineskip}

\title{\bf Sensitivity Analysis Without Assumptions}
\author{Peng Ding and Tyler VanderWeele\\
Harvard University\\
Emails: \url{pengdingpku@gmail.com} and \url{tvanderw@hsph.harvard.edu}
}
\date{}

\maketitle

\begin{center}
\bfseries Abstract
\end{center}
Unmeasured confounding may undermine the validity of causal inference with observational studies. Sensitivity analysis provides an attractive way to partially circumvent this issue by assessing the potential influence of unmeasured confounding on the causal conclusions. However, previous sensitivity analysis approaches often make strong and untestable assumptions such as having a confounder that is binary, or having no interaction between the effects of the exposure and the confounder on the outcome, or having only one confounder. Without imposing any assumptions on the confounder or confounders, we derive a bounding factor and a sharp inequality such that the sensitivity analysis parameters must satisfy the inequality if an unmeasured confounder is to explain away the observed effect estimate or reduce it to a particular level. Our approach is easy to implement and involves only two sensitivity parameters. Surprisingly, our bounding factor, which makes no simplifying assumptions, is no more conservative than a number of previous sensitivity analysis techniques that do make assumptions. Our new bounding factor implies not only the traditional Cornfield conditions that both the relative risk of the exposure on the confounder and that of the confounder on the outcome must satisfy, but also a high threshold that the maximum of these relative risks must satisfy. Furthermore, this new bounding factor can be viewed as a measure of the strength of confounding between the exposure and the outcome induced by a confounder. 


\section{Introduction}
Causal inference with observational studies is of great interest and importance in many scientific disciplines. Although unmeasured confounding between the exposure and the outcome may bias the estimation of the true causal effect, an approach often called ``sensitivity analysis'' or ``bias analysis'' over a range of sensitivity parameters sometimes allows researchers to make causal inferences even without full control of the confounders of the relationship between the exposure and outcome.

Sensitivity analysis plays a central role in assessing the influence of the unmeasured confounding on the causal conclusions. However, many sensitivity analysis techniques often require additional untestable assumptions. For instance, some authors assume a single binary confounder \cite{Cornfield::1959, Bross::1966, Bross::1967, Yanagawa::1984, Rosenbaum::1983JRSSB, Imbens::2003}. Researchers also often assume a homogeneity assumption that there is no interaction between the effects of the exposure and the confounder on the outcome \cite{Schlesselman::1978, Vanderweele::2011, Rosenbaum::1983JRSSB, Lin::1998, Imbens::2003}. Some sensitivity analysis techniques only allow one to assess how strong an unmeasured confounder would have to be to completely explain away an effect \cite{Cornfield::1959, Bross::1966, Bross::1967, Lee::2011, Ding::2014}, but do not allow one to assess what the effect estimate might be under weaker unmeasured confounding scenarios, i.e., do not allow one to do sensitivity analysis under alternative hypotheses. Performing sensitivity analysis under alternative hypotheses can be quite challenging due to more parameters needed in the sensitivity analysis. Cornfield et al.'s early work \citep{Cornfield::1959} on sensitivity analysis for the cigarette smoking and lung cancer association, which helped initiate the entire field of sensitivity analysis, in fact made all three simplifying assumptions: a single binary confounder, no interaction, and only sensitivity analysis for the null hypothesis of no causal effect.
Although some sensitivity analysis results exist for general confounders \citep{Vanderweele::2011, Flanders::1990}, they are only easy to implement under some of the above simplifying assumptions.

In this paper, we propose a new bounding factor and sensitivity analysis technique without any assumptions about the unmeasured confounder or confounders. None of the null hypothesis, a single binary confounder, or the no-interaction assumption is required for using the bounding factor. Nonetheless, our new bounding factor, which makes no simplifying assumptions, is no more conservative than many previous sensitivity analysis techniques that do make assumptions and is furthermore easy to implement. Moreover, we show that the new bounding factor implies not only the classical Cornfield conditions \cite{Cornfield::1959} that both the relative risk of the exposure on the confounder and that of the confounder on the outcome must satisfy, but also a stronger condition that the maximum of these relative risks must satisfy. The new bounding factor can be viewed as a measure of the strength of confounding between the exposure and the outcome resulting from the confounder.
We begin by considering outcomes which are binary and extend our results further to time-to-event and non-negative count or continuous outcomes. We consider both ratio and difference scales.

The claim that our technique is ``without assumptions'' requires some clarification. As we will see below, we will, without any assumptions, be able to make statements of the form: ``For an observed association to be due solely to unmeasured confounding, two sensitivity analysis parameters must satisfy [a specific inequality].'' We will also, without assumptions, be able to make statements of the form: ``For unmeasured confounding alone to be able to reduce an observed association [to a given level], two sensitivity analysis parameters must satisfy [another specific inequality].'' We believe the ability to make statements of this form without imposing any specific structure on the nature of the unmeasured confounder or confounders constitutes a major advance in the literature.

However, if statements are made of the form, ``If the sensitivity analysis parameter take [specified values], then such unmeasured confounding can reduce the observed estimate by no more than [a specific level],'' then the specification of the sensitivity analysis parameters could itself of course be viewed as an assumption. Moreover, when placing the results within a counterfactual or potential outcomes framework, the assumptions implicit within that framework of course would be needed also to give a potential outcomes interpretation to the sensitivity analysis. Thus certain types of statements concerning the sensitivity of conclusions to unmeasured confounding can be made ``without assumptions,'' while other types of statements do require assumptions concerning the specification of the sensitivity analysis parameters themselves, or those implicit within the potential outcomes framework.

Our title perhaps merits one further qualification which is that what is called in this paper ``sensitivity analysis'' is generally now referred to as ``bias analysis'' in the epidemiologic literature. Moreover, such ``bias analysis'' is relevant not only to problems of unmeasured confounding but also measurement error and selection bias, and our focus in this paper only concerns unmeasured confounding. The term ``sensitivity analysis'' is,  however, still employed in statistics, econometrics, and in many of the social sciences for issues of unmeasured confounding. We believe the technique presented in this paper will be useful across this range of disciplines and have chosen to use the broader term, while acknowledging that terminology in epidemiology has shifted.

\section{Main Result: A New Bounding Factor}
Let $E$ denote the exposure, $D$ denote a binary outcome, $C$ denote the measured confounders, and $U$ denote one or more unmeasured confounders. We will assume for what follows that the exposure $E$ is binary, but all of the results below are also applicable to a categorical or continuous exposure and could be applied comparing any two levels of $E$.  For ease of notation, we assume that the unmeasured confounder $U$ is categorical with levels $0,1,\ldots, K-1$. But all the conclusions hold for $U$ of general type (categorical, continuous, or mixed; single or multiple confounders). We provide proofs and theoretical technical details for general $U$ in the eAppendix.

Let $\RR_{ED|c}^\obs =  \pr(D=1\mid E=1, C=c) /   \pr(D=1\mid E=0,C=c)$ denote the observed relative risk of the exposure $E$ on the outcome $D$ within stratum of measured confounders $C=c$. Define $\GRR_{EU,k|c}  =\pr(U=k\mid E=1, C=c) / \pr(U=k\mid E=0, C=c)$ as the relative risk of exposure on category $k$ of the unmeasured confounder within stratum of measured confounders $C=c$. We use $\GRR_{EU|c} = \max_k \GRR_{EU, k|c}  $ to denote the maximum of these relative risks between $E$ and $U$, which we will call the maximal relative risk of $E$ on $U$ within stratum $C=c$. Define 
$$
\GRR_{UD|E=0,c}  ={  \max_k \pr(D=1\mid E=0, C=c, U=k) \over  \min_k \pr(D=1\mid E=0, C=c, U=k) }
$$ 
as the maximum of the effect of $U$ on $D$ among the unexposed comparing any two categories of $U$ (i.e., the ratio of the maximum and minimum of the probabilities of the outcome over strata of $U$ without exposure and within stratum $C=c$); similarly, define 
$$ 
\GRR_{UD|E=1,c} = { \max_k \pr(D=1\mid E=1, C=c, U=k) \over  \min_k \pr(D=1\mid E=1, C=c, U=k) }
$$ 
as the maximum of the effect of $U$ on $D$ among the exposed comparing any two categories of $U$ (i.e., the ratio of the maximum and minimum of the probabilities of the outcome over strata of $U$ with exposure and within stratum $C=c$). We use $\GRR_{UD|c} = \max(  \GRR_{UD|E=1,c}, \GRR_{UD|E=0,c} )$ to  denote the maximum of the relative risks between $U$ and $D$ with and without exposure, defined as the maximal relative risk of $U$ on $D$ within stratum $C=c$. Note that if $U$ is a vector that contains multiple unmeasured confounders, then $\GRR_{EU|c}$ and $\GRR_{UD|c}$ are defined as the maximum relative risk comparing any two categories of the vector $U$.

If $C$ and $U$ suffice to control for confounding for the effect of $E$ on $D$, the standardized relative risk
\begin{eqnarray*}
\SRR_{ED|c} = { \sum_{k=0}^{K-1} \pr(D=1\mid E=1, C=c, U=k)  \pr(U=k\mid C=c)   \over    
\sum_{k=0}^{K-1} \pr(D=1\mid E=0, C=c, U=k)  \pr(U=k\mid C=c) }
\end{eqnarray*}
is the true causal relative risk of the exposure $E$ on the outcome $D$ within stratum $C=c.$ In the main text, we focus the discussion on the whole population. We further show in the eAppendix that all the conclusions also hold for exposed and unexposed subpopulations. 

We will for the next several sections assume all analyses are carried out within strata of $C$, and thus the condition $C=c$ is omitted and kept implicitly in all the conditional probabilities, e.g., $\RR_{ED|c}^\obs $ is replaced by $\RR_{ED}^\obs $ for notational simplicity. Later in the paper we will comment on how the results are applicable to estimation averaged over $C$, rather than conditional on $C.$

The relative risk pair $(\GRR_{EU}, \GRR_{UD})$ measures the strength of confounding between the exposure $ E$ and the outcome $D$ induced by the confounder $U$. Our main result ties the ratio of the observed relative risk $\RR_{ED}^\obs $ adjusted only for measured confounders $C$ and the true relative risk $\SRR_{ED}$ adjusted also for unmeasured confounders $U$, to the strength of confounding, $(\GRR_{EU}, \GRR_{UD})$. Without any assumptions, we have the following result:

\noindent {\bfseries Result 1.} 
\begin{eqnarray}
\label{eq::main-result}
\SRR_{ED} \geq \RR_{ED}^\obs  \Big/  \frac{  \GRR_{EU} \times \GRR_{UD}   }{    \GRR_{EU} +  \GRR_{UD}  - 1}  .
\end{eqnarray}
Result 1 shows that even in the presence of unmeasured confounding the true relative risk must be at least as large as $\RR_{ED}^\obs  \Big/  \frac{  \GRR_{EU} \times \GRR_{UD}   }{    \GRR_{EU} +  \GRR_{UD}  - 1}.$ In the eAppendix, we provide a proof for result (\ref{eq::main-result}) and also show that the inequality is sharp in the sense that we can always construct a model with a confounder $U$ to attain the equality. The quantity $ (\GRR_{EU} \times \GRR_{UD} )  / (  \GRR_{EU} +  \GRR_{UD}  - 1)  $ is a new joint bounding factor for the relative risk. Although quite simple, this bound using both $\GRR_{EU}$ and $\GRR_{UD}$ has several important implications.

First, the result essentially allows for sensitivity analysis without assumptions insofar as for an unmeasured confounder to reduce an observed estimated $\RR_{ED}^\obs$ to an actual relative risk of $\SRR_{ED}$ the sensitivity analysis parameters $\GRR_{EU}$ and $\GRR_{UD}$ must be sufficiently large to satisfy the inequality 
$$
\frac{  \RR_{EU} \times  \RR_{UD} } { \RR_{EU} + \RR_{UD} -1} \geq  \frac{ \RR_{ED}^\obs }{  \SRR_{ED} }.
$$  
This statement holds without any assumptions about the nature of the unmeasured confounder. One could plot those values of $\RR_{EU}$ and $\RR_{UD}$ that would be required to explain away the effect estimate (or the lower limit of a confidence interval). To conduct sensitivity analysis with pre-specified strength of the unmeasured confounder, $(\GRR_{EU}, \GRR_{UD})$, we can divide the observed relative risk and its confidence limits by $ (\GRR_{EU} \times \GRR_{UD} )  / (  \GRR_{EU} +  \GRR_{UD}  - 1)  $, in order to obtain a point estimate and confidence limits of the lower bound of the true causal effect of the exposure $E$ on the outcome $D$. We will refer to the relative risk adjusted only for $C$, when divided by the bounding factor $ (\GRR_{EU} \times \GRR_{UD} )  / (  \GRR_{EU} +  \GRR_{UD}  - 1)  $ as the corrected relative risk. It is ``corrected'' in the sense that an unmeasured confounder cannot reduce the relative risk any further than what is obtained by division by its bounding factor. As an example, suppose we have an observed relative risk of $2.1$ with a $95\%$ confidence interval $[1.4, 3.1]$. If we consider an unmeasured confounder with $(\RR_{EU}, \RR_{UD})=(2,2)$, then the joint bounding factor is $2\times2/(2+2-1)=1.33$, and the corrected relative risk is $2.1/1.33=1.58$ with a $95\%$ confidence interval $[1.4/1.33,3.1/1.33] = [1.05, 2.33]$. Therefore, the confounder with $(\RR_{EU}, \RR_{UD})=(2,2)$ cannot explain away the observed relative risk $2.1$ or its lower confidence limit $1.4$, i.e., it cannot reduce the point estimate and lower confidence limit of the relative risk to be smaller than one.
If we consider an unmeasured confounder with $(\RR_{EU}, \RR_{UD})=(2.5, 3.5)$, then the joint bounding factor is $2.5\times 3.5/(2.5+3.5-1)=1.75$, and an estimate for the lower bound of the true causal relative risk is $2.1/1.75=1.20$ with a $95\%$ confidence interval $[1.4/1.75, 3.1/1.75] = [0.8, 1.77]$. Although the confounder with $(\RR_{EU}, \RR_{UD})=(2.5,3.5)$ cannot explain away the observed relative risk of $2.1$, it reduces the original lower confidence limit $1.4$ to $0.8$ (i.e., less than one).
Note that we are not merely assessing a binary confounder, and we are not imposing the no interaction assumption. Moreover, we are not restricted to only assessing how much confounding can explain away an effect, nor are we even assuming that there is a single unmeasured confounder (since $U$ can be a vector of unmeasured confounders). The corrected estimates and confidence intervals above are applicable irrespective of the underlying confounder (or confounders). We can apply the technique to obtain a range of values for the true causal effect under different specifications of $\GRR_{EU}$ and $\GRR_{UD}$.

Table \ref{tb::biasfactor} shows the magnitudes of the joint bounding factor for different combinations of $\GRR_{EU}$ and $\GRR_{UD}$. The entries in the table for the joint bounding factor are the largest observed relative risks that such an unmeasured confounder could explain away. We can see from the table that the joint bounding factor is always smaller than both of $\GRR_{EU}$ and $\GRR_{UD}$, and much smaller than the maximum of them.

\begin{table}[ht]
\centering
\caption{Magnitudes of the joint bounding factor for different combinations of $\GRR_{EU}$ and $\GRR_{UD}$}\label{tb::biasfactor}
\resizebox{\columnwidth}{!}{
\begin{tabular}{|rr|rrrrrrrrrrrr|}
  \hline
 & & \multicolumn{12}{c|}{$\GRR_{UD}$}\\
\multicolumn{2}{|c|}{bounding factor}  & 1.3 & 1.5 & 1.8 & 2 & 2.5 & 3 & 3.5 & 4 & 5 & 6 & 8 & 10 \\ 
  \hline 
\multirow{12}{*}{$\GRR_{EU}$} & 1.3 & 1.06 & 1.08 & 1.11 & 1.13 & 1.16 & 1.18 & 1.20 & 1.21 & 1.23 & 1.24 & 1.25 & 1.26 \\ 
&  1.5 & 1.08 & 1.12 & 1.17 & 1.20 & 1.25 & 1.29 & 1.31 & 1.33 & 1.36 & 1.38 & 1.41 & 1.43 \\ 
&  1.8 & 1.11 & 1.17 & 1.25 & 1.29 & 1.36 & 1.42 & 1.47 & 1.50 & 1.55 & 1.59 & 1.64 & 1.67 \\ 
&  2 & 1.13 & 1.20 & 1.29 & 1.33 & 1.43 & 1.50 & 1.56 & 1.60 & 1.67 & 1.71 & 1.78 & 1.82 \\ 
&  2.5 & 1.16 & 1.25 & 1.36 & 1.43 & 1.56 & 1.67 & 1.75 & 1.82 & 1.92 & 2.00 & 2.11 & 2.17 \\ 
&  3 & 1.18 & 1.29 & 1.42 & 1.50 & 1.67 & 1.80 & 1.91 & 2.00 & 2.14 & 2.25 & 2.40 & 2.50 \\ 
&  3.5 & 1.20 & 1.31 & 1.47 & 1.56 & 1.75 & 1.91 & 2.04 & 2.15 & 2.33 & 2.47 & 2.67 & 2.80 \\ 
&  4 & 1.21 & 1.33 & 1.50 & 1.60 & 1.82 & 2.00 & 2.15 & 2.29 & 2.50 & 2.67 & 2.91 & 3.08 \\ 
&  5 & 1.23 & 1.36 & 1.55 & 1.67 & 1.92 & 2.14 & 2.33 & 2.50 & 2.78 & 3.00 & 3.33 & 3.57 \\ 
&  6 & 1.24 & 1.38 & 1.59 & 1.71 & 2.00 & 2.25 & 2.47 & 2.67 & 3.00 & 3.27 & 3.69 & 4.00 \\ 
&  8 & 1.25 & 1.41 & 1.64 & 1.78 & 2.11 & 2.40 & 2.67 & 2.91 & 3.33 & 3.69 & 4.27 & 4.71 \\ 
&  10 & 1.26 & 1.43 & 1.67 & 1.82 & 2.17 & 2.50 & 2.80 & 3.08 & 3.57 & 4.00 & 4.71 & 5.26 \\ 
   \hline
\end{tabular}
}
\end{table}

As a second important consequence of our main result in (\ref{eq::main-result}), we also show in the eAppendix that once we specify one of the unmeasured confounding measures, for example $\GRR_{EU}$, then to be able to reduce an observed relative risk of $\RR_{ED}^\obs $ to a true causal relative risk of $\SRR_{ED}$ the other confounding measure $\RR_{UD}$ must be at least of the magnitude
$$
 \GRR_{UD} \geq   \frac{\GRR_{EU}\times  \RR_{ED}^\obs   -  \RR_{ED}^\obs   }{ \GRR_{EU} \times \SRR_{ED} - \RR_{ED}^\obs } .
$$
For an unmeasured confounder to completely explain away the relative risk, i.e., reduce $\RR_{ED}^\obs $ to $\SRR_{ED}=1$, once we specify $\GRR_{EU}$ the other unmeasured confounding measure much be at least of the magnitude
$$
\GRR_{UD} \geq  \frac{     \GRR_{EU} \times \RR_{ED}^\obs   -  \RR_{ED}^\obs     }{    \GRR_{EU}   - \RR_{ED}^\obs    }.
$$
For example, if we have an observed relative risk $\RR_{ED}^\obs  = 2.5$, and we specify the exposure-confounder association $\GRR_{EU} = 3$. Then in order to reduce the observed relative risk to a true causal relative risk $\SRR_{ED}=1.5$, the confounder-outcome association must be at least as large as $ (3\times 2.5 - 2.5)/(3\times 1.5 - 2.5) = 2.5$; in order to completely explain away the observed relative risk (i.e., to reduce the observed relative risk to $\SRR_{ED} = 1$), the confounder-outcome association must be at least as large as $ (3\times 2.5 - 2.5)/(3 - 2.5) =10.$
The symmetry of result (\ref{eq::main-result}) implies that a similar result also holds for $\GRR_{EU}$ with pre-specified $ \GRR_{UD} $.

Third, we show in the eAppendix that if both the generalized relative risks $\GRR_{EU}$ and $ \GRR_{UD}$ have the same magnitude, for an unmeasured confounder to reduce an observed relative risk of $\RR_{ED}^\obs $ to a true causal relative risk of $\SRR_{ED}$ both of the confounding relative risks must thus be at least as large as
$$
\GRR_{EU} =  \GRR_{UD} \geq \left\{   \RR_{ED}^\obs  + \sqrt{   \RR_{ED}^\obs  (\RR_{ED}^\obs  - \SRR_{ED})     }  \right\} \Big/ \SRR_{ED}.
$$
For an unmeasured confounder to completely explain away an observed relative risk of $\RR_{ED}^\obs $ (i.e., to reduce $\RR_{ED}^\obs $ to a true causal relative risk of $\SRR_{ED}=1$), both $\GRR_{EU}$ and $ \GRR_{UD}$ must be at least as large as
$$
\GRR_{EU} =  \GRR_{UD} \geq    \RR_{ED}^\obs  + \sqrt{   \RR_{ED}^\obs  (\RR_{ED}^\obs  -1)     }  .
$$
If one of the confounding relative risks is smaller than the lower bound above, we then know that the other one must be larger. Thus even if $\GRR_{EU}$ and $\GRR_{UD}$ are not of the same magnitude, the maximum of $\GRR_{EU}$ and $\GRR_{UD}$ must satisfy the inequality above. We then have the following ``high threshold'' condition:
$$
\max ( \GRR_{EU}, \GRR_{UD} ) \geq \left\{   \RR_{ED}^\obs  + \sqrt{   \RR_{ED}^\obs  (\RR_{ED}^\obs  - \SRR_{ED})     }  \right\} \Big/ \SRR_{ED}. 
$$
For example, in order to reduce an observed relative risk of $\RR_{ED}^\obs  = 2.5$ to a true causal relative risk of $\SRR_{ED}=1.5$, the high threshold is $(2.5 + \sqrt{2.5\times 1})/1.5 = 2.72$; at least one of $\RR_{EU}$ and $\RR_{UD}$ must be of magnitude $2.72$ or above. In order to completely explain away an observed relative risk of $\RR_{ED}^\obs  = 2.5$, the high threshold is $2.5+\sqrt{2.5\times 1.5} = 4.44$; at least 
one of $\RR_{EU}$ and $\RR_{UD}$ must be of magnitude $4.44$ or higher to completely explain away the effect.

Fourth, the bias formula in (\ref{eq::main-result}) is relevant for an apparently causative exposure, which allows researchers to get lower bounds of the true causal relative risk given pre-specified sensitivity parameters $\GRR_{EU}$ and $\GRR_{UD}$. If the exposure $E$ is apparently preventive with $\RR_{ED}^\obs  < 1$, we can use the following formula to conduct sensitivity analysis:
\begin{eqnarray}
\SRR_{ED}  \leq   \RR_{ED}^\obs  \times  \frac{ \GRR_{EU}\times \GRR_{UD} }{  \GRR_{EU} +  \GRR_{UD}  - 1 },
\label{eq::main-preventive}
\end{eqnarray} 
where we modify the definition of $\GRR_{EU}$ as $\max_k \GRR_{EU, k}^{-1}$, i.e., the maximum of the inverse relative risks relating $E$ and $U$, or equivalently the inverse of the minimum of the relative risks relating $E$ and $U.$
For an apparently preventive exposure, (\ref{eq::main-preventive}) allows researchers to obtain an upper bound of the causal relative risk $\SRR_{ED}$ by multiplying the observed relative risk $ \RR_{ED}^\obs  $ by the joint bounding factor $\GRR_{EU}\times \GRR_{UD} /( \GRR_{EU} +  \GRR_{UD}  - 1)$. We present the proof in the eAppendix, and omit analogous discussion based on (\ref{eq::main-preventive}).

Finally, all the results above are within strata of the observed covariates $C$ as would be obtained from a log-binomial regression model or a logistic regression model with rare outcome. If averaged relative risk over the observed covariates $C$ is of interest, the true causal relative risk must be at least as large as the minimum of $\RR_{ED|c}^\obs \Big/  \frac{  \GRR_{EU|c} \times \GRR_{UD|c}   }{    \GRR_{EU|c}+  \GRR_{UD|c}  - 1}$ over $c$. If we assume a common causal relative risk among the levels of $C$ as in the usual log-linear or logistic regression with rare outcomes, then the true causal relative risk must be at least as large as the maximum of $\RR_{ED|c}^\obs
\Big/  \frac{  \GRR_{EU|c} \times \GRR_{UD|c}   }{    \GRR_{EU|c}+  \GRR_{UD|c}  - 1}$ over $c$.
See the eAppendix for further discussion.

\section{Relation with Cornfield Conditions}

Under the assumptions of a binary confounder $U$ and the conditional independence between the exposure $E$ and the outcome $D$ given the confounder $U$, Cornfield et al. \cite{Cornfield::1959} showed that the exposure-confounder relative risk must be at least as large as the observed exposure-outcome relative risk:
\begin{eqnarray}
\RR_{EU} \geq \RR_{ED}^\obs .
\label{eq::corn-1}
\end{eqnarray}
Schlesselman \cite{Schlesselman::1978} further showed that the confounder-outcome relative risk must also be at least as large as the observed exposure-outcome relative risk:
\begin{eqnarray}
\RR_{UD} \geq \RR_{ED}^\obs .
\label{eq::corn-2}
\end{eqnarray}
We show in the eAppendix that the classical Cornfield conditions (\ref{eq::corn-1}) and (\ref{eq::corn-2}) are just special cases of our result by letting one of $\RR_{EU}$ or $\RR_{UD}$ go to infinity in (\ref{eq::main-result}). Moreover, our results apply to general confounders not just binary confounders, and our results also apply to other possible values of the true causal relative risk of the exposure on the outcome. We are not restricted to only assessing how strong the unmeasured confounder would have to be to completely explain away the effect. Thus, for example, for confounding to reduce the observed relative risk $\RR_{ED}^\obs $ to a true causal relative risk of $\SRR_{ED}$, the unmeasured confounding measures have to satisfy 
\begin{eqnarray}
\GRR_{EU} \geq \RR_{ED}^\obs  / \SRR_{ED} \quad \text{ and } \quad \GRR_{UD} \geq \RR_{ED}^\obs  / \SRR_{ED}.
\label{eq::corn-12}
\end{eqnarray}
Perhaps even more importantly with regard to Cornfield-like conditions, our main result in (\ref{eq::main-result}) not only leads to the conditions in (\ref{eq::corn-12}) that both $\GRR_{EU}$ and $\GRR_{UD}$ must satisfy, but also implies the following condition that the maximum of $\GRR_{EU}$ and $\GRR_{UD}$ must satisfy:
\begin{eqnarray}
\max(\GRR_{EU}, \GRR_{UD}) \geq \left\{   \RR_{ED}^\obs  + \sqrt{   \RR_{ED}^\obs  (\RR_{ED}^\obs  - \SRR_{ED})     }  \right\} \Big/ \SRR_{ED},
\label{eq::corn-high}
\end{eqnarray}
to reduce an observed relative risk $\RR_{ED}^\obs $ to a true causal relative risk $\SRR_{ED}.$ We show this in the eAppendix. As a special case, for the unmeasured confounder to completely explain away the observed relative risk (i.e., $\SRR_{ED} = 1$), it is necessary that 
$$
\max(\RR_{EU}, \RR_{UD} ) \geq \RR_{ED}^\obs  + \sqrt{  \RR_{ED}^\obs  (\RR_{ED}^\obs  - 1) }.
$$ 
Once again the results do not require a binary unmeasured confounder. They are applicable to any unmeasured confounder. Similar low and high threshold Cornfield conditions that the minimum and maximum of the confounding measures must satisfy to completely explain away an effect were derived on an odds ratio scale of exposure-confounder association by Flanders and Khoury \citep{Flanders::1990} and Lee \citep{Lee::2011}, and we comment and extend these results in the eAppendix.

The classical Cornfield conditions and the high threshold generalization are useful to answer the question about the magnitude of the association between the exposure and the confounder and that between the confounder and the outcome, in order to explain away the observed exposure-outcome association or with our new results, to reduce it to a pre-specified magnitude. The Cornfield conditions in (\ref{eq::corn-12}) and (\ref{eq::corn-high}) are especially useful, when we want to specify only one of the marginal associations $\GRR_{EU}$ or $\GRR_{UD}$ as well as their relative magnitudes. However, they are inferior to the main result in (\ref{eq::main-result}), which is essentially the condition that the joint values of $(\GRR_{EU}, \GRR_{UD} )$ must satisfy. As will be seen below, although the high threshold conclusions are a useful heuristic, they are weaker than the use of our new joint bounding factor in (\ref{eq::main-result}) insofar as there are scenarios which the joint bounding factor in (\ref{eq::main-result}) can rule out an estimate as being due to unmeasured confounding but the high threshold conditions cannot.
For example, when we have an observed exposure-outcome relative risk of $\RR_{ED}^\obs  = 3$, the low threshold (i.e., the classical Cornfield condition) is given by
$$
\min(\GRR_{EU}, \GRR_{UD} ) \geq 3,
$$
the high threshold is given by
$$
\max(\GRR_{EU}, \GRR_{UD} ) \geq 3+\sqrt{3\times 2} =  5.45,
$$
and the joint threshold condition is given by
$$
\frac{  \GRR_{EU} \times \GRR_{UD}   }{    \GRR_{EU} +  \GRR_{UD}  - 1} \geq 3.
$$
Thus, the low Cornfield threshold is $3$, and so we know that we must have that both $\GRR_{EU}$ and $\GRR_{UD}$ be greater than $3$ to explain away the effect. The high Cornfield threshold is $5.45$, and so at least one of $\GRR_{EU}$ and $\GRR_{UD}$ must be larger than $5.45$ to explain away the effect. Consider an unmeasured confounder with $(\RR_{EU}=5.5, \RR_{UD}=3.1)$, they would exceed both the low Cornfield threshold (since $\RR_{EU}>3, \RR_{UD}>3$) and the high threshold (since $\RR_{EU}>5.45$), and we might thus think it can explain away the observed exposure-outcome relative risk. However, using our joint threshold condition in (\ref{eq::main-result}), an unmeasured confounder with $(\RR_{EU}=5.5, \RR_{UD}=3.1)$ has a bounding factor $5.5\times 3.1/(5.5+3.1-1) = 2.24<3$ and thus such confounding could not explain away an observed relative risk of $3$. We can see this from our result in (\ref{eq::main-result}), but we cannot see this from the classical Cornfield conditions and even the new high threshold Cornfield condition. The Cornfield conditions, both low and high thresholds, although a useful heuristic, are not as useful for sensitivity analysis as our bounding factor in (\ref{eq::main-result}) insofar as there as scenarios, such as the one above, which our bounding factor in (\ref{eq::main-result}) can rule out an estimate as being due to unmeasured confounding but the low and high threshold Cornfield conditions cannot.

\section{Illustration}

Consider the historical study conducted by Hammond and Horn \citep{Hammond::1968}, in which the point estimate of the observed relative risk of cigarette smoking on lung cancer was $\RR_{ED}^\obs =10.73$ with $95\%$ confidence interval $[8.02,14.36]$. Fisher \cite{Fisher::1957} suggested that the observed relative risk of the exposure $E$ on the outcome $D$ might be completely due to the existence of a common genetic confounder. The work of Cornfield et al. \citep{Cornfield::1959} showed that for a binary unmeasured confounder to completely explain away the observed relative risk, both the exposure-confounder relative risk and the confounder-outcome relative risk would have to be at least $10.73$. Let us now assume then that both the exposure-confounder relative risk and the confounder-outcome relative risk have the magnitude $10.73$. The joint bounding factor is 
$$
\frac{  \GRR_{EU} \times \GRR_{UD}   }{    \GRR_{EU} +  \GRR_{UD}  - 1} =  \frac{10.73\times 10.73}{10.73+10.73-1}  = 5.63.
$$
Even if we assume such a strong confounder, the point estimate of the causal relative risk of cigarette smoking and lung cancer must still be at least as large as 
$
\SRR_{ED} \geq \RR_{ED}^\obs  / 5.63
=   10.73 / 5.63 = 1.91 > 1,
$
and the $95\%$ confidence interval is $[8.02/5.63,14.36/5.63] =   [ 1.42, 2.55 ]$ with the lower confidence limit still larger than one. Thus in fact, not even exposure-confounder and confounder-outcome relative risks of $10.73$ suffice to explain away the effect nor the lower confidence limit.
In fact, in order to explain away the point estimate of the observed relative risk $10.73$, the magnitude of $\GRR_{EU}$ and $\GRR_{UD}$ (if $\GRR_{EU} = \GRR_{UD}$) should be at least as large as $10.73+\sqrt{10.73\times 9.73} = 20.95$. And in order to explain away the lower confidence limit $8.02$, these two confounding relative risks should be at least as large as $8.02+\sqrt{8.02\times 7.02} = 15.52.$ More generally, we can plot those values of $\RR_{EU}$ and $\RR_{UD}$ that would be required to explain away the effect estimate or the lower limit of the confidence interval. This is given in Figure \ref{fg::jointvalueplot}.
To explain away the point estimate the two parameters would have to lie on or above the solid line. To explain away the lower confidence limit the two parameters would have to lie on or above the dotted line. These results hold without any assumptions on the structure of the unmeasured confounding.
The numerical results above show that by using the new joint bounding factor it is even more implausible than using the Cornfield conditions that a genetic confounder explains away the relative risk between cigarette smoking and lung cancer.

\begin{figure}[h]
\centering
\includegraphics[width=0.8\textwidth]{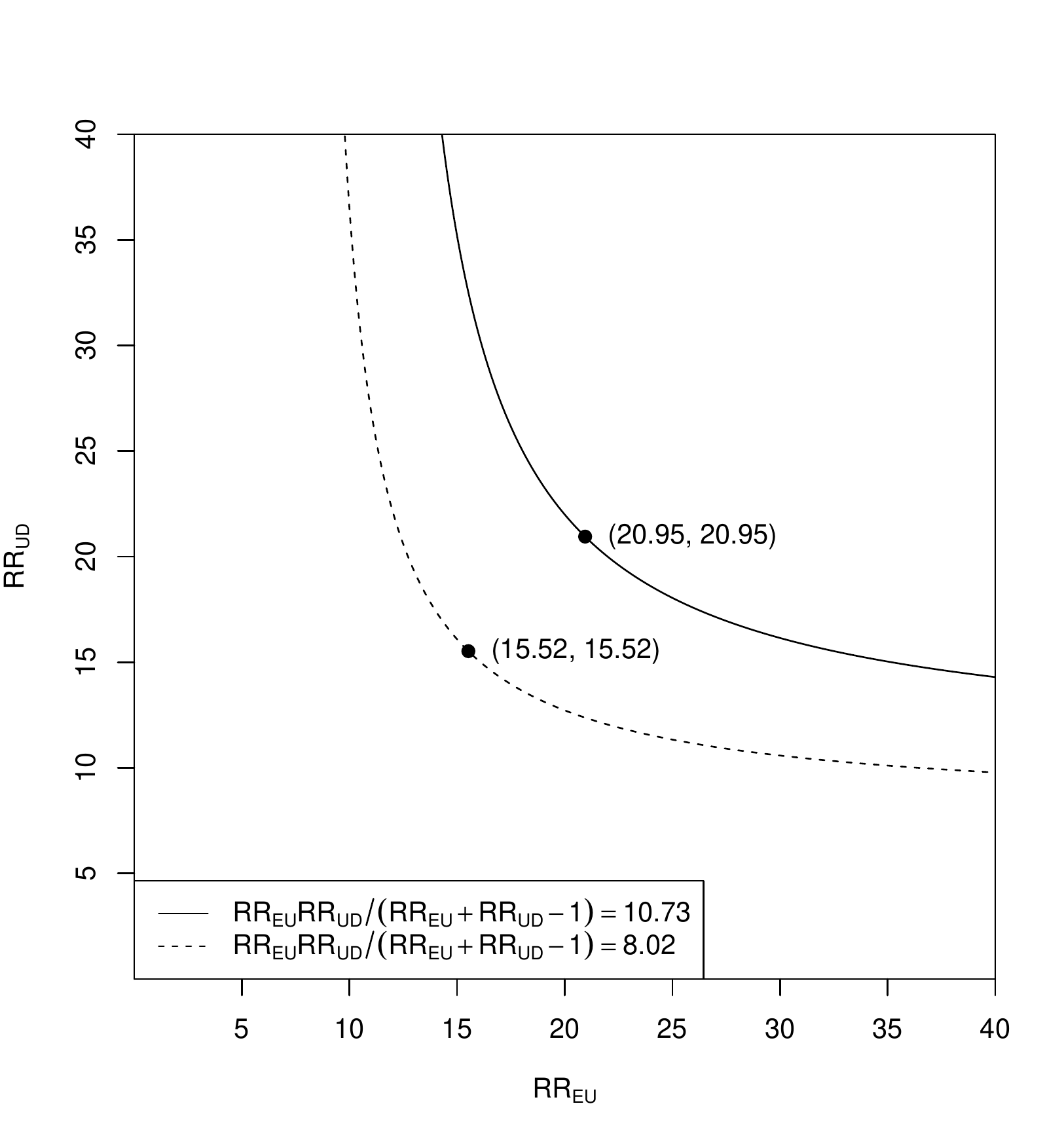}
\caption{The areas above the two lines are the joint values of $( \RR_{EU}, \RR_{UD})$ that can would be required to explain away the effect estimate $10.73$ and the lower confidence limit $8.02$.}\label{fg::jointvalueplot}
\end{figure}

More generally, we could consider corrected estimates and confidence intervals for the effect over a range of different values of the sensitivity analysis parameters, $\RR_{EU}$ and $\RR_{UD}$, as in Table \ref{tb::table}. The columns of Table \ref{tb::table} correspond to $\RR_{UD}$ and the rows to $\RR_{EU}$. The entries are the corrected estimates and confidence intervals for the effect under the different confounding scenarios. In general a table like this one is most informative for sensitivity analysis. SAS code to carry out such a sensitivity analysis and to provide such a table is given in Appendix 1.

\begin{table}[h]
\centering
\caption{Bounds on corrected estimates, lower confidence limits, and upper confidence limits for unmeasured confounding (each cell contains bounds on point estimate, lower and upper confidence limits; columns correspond to increasing strength of the risk ratio of $U$ on the outcome; rows correspond to increasing strength of risk ratio relating the exposure and $U$)}
\label{tb::table}
\resizebox{\columnwidth}{!}{
\begin{tabular}{|r|r|r|r|r|r|r|r|r|r|r|}
  \hline
 & 1.2 & 1.3 & 1.5 & 1.8 & 2.0 &2.5 & 3.0 & 5.0 & 8.0 &  10.0 \\ 
  \hline
1.2 & 10.43 & 10.32 & 10.13 & 9.94 & 9.84 & 9.66 & 9.54 & 9.30 & 9.17 & 9.12 \\ 
  & (7.80, & (7.71, & (7.57, & (7.43, & (7.35, & (7.22, & (7.13, & (6.95, & (6.85, & (6.82, \\ 
  & 13.96) & 13.81) & 13.56) & 13.30) & 13.16) & 12.92) & 12.76) & 12.45) & 12.27) & 12.21) \\ 
  \hline 
 1.3 & 10.32 & 10.16 & 9.90 & 9.63 & 9.49 & 9.24 & 9.08 & 8.75 & 8.56 & 8.50 \\ 
  & (7.71, & (7.59, & (7.40, & (7.20, & (7.09, & (6.91, & (6.79, & (6.54,& (6.40, & (6.35, \\ 
  & 13.81) & 13.60) & 13.26) & 12.89) & 12.70) & 12.37) & 12.15) & 11.71) & 11.46) & 11.38) \\ 
  \hline 
 1.5 & 10.13 & 9.90 & 9.54 & 9.14 & 8.94 & 8.58 & 8.35 & 7.87 & 7.60 & 7.51 \\ 
   & (7.57, & (7.40, & (7.13, & (6.83, & (6.68, & (6.42, & (6.24, & (5.88, & (5.68, & (5.61, \\ 
   & 13.56) & 13.26) & 12.76) & 12.23) & 11.97) & 11.49) & 11.17) & 10.53) & 10.17) & 10.05) \\ 
   \hline 
 1.8 & 9.94 & 9.63 & 9.14 & 8.61 & 8.35 & 7.87 & 7.55 & 6.91 & 6.56 & 6.44 \\ 
   & (7.43, & (7.20, & (6.83, & (6.44, & (6.24, & (5.88, & (5.64, & (5.17, & (4.90, & (4.81, \\ 
   & 13.30) & 12.89) & 12.23) & 11.52) & 11.17) & 10.53) & 10.11) & 9.25) & 8.78) & 8.62) \\ 
   \hline 
 2.0 & 9.84 & 9.49 & 8.94 & 8.35 & 8.05 & 7.51 & 7.15 & 6.44 & 6.04 & 5.90 \\ 
   & (7.35, & (7.09, & (6.68, & (6.24, & (6.01, & (5.61, & (5.35, & (4.81, & (4.51, & (4.41, \\ 
   & 13.16) & 12.70) & 11.97) & 11.17) & 10.77) & 10.05) & 9.57) & 8.62) & 8.08) & 7.90) \\ 
   \hline 
 2.5 & 9.66 & 9.24 & 8.58 & 7.87 & 7.51 & 6.87 & 6.44 & 5.58 & 5.10 & 4.94 \\ 
   & (7.22, & (6.91,& (6.42, & (5.88, & (5.61,& (5.13, & (4.81, & (4.17, & (3.81, & (3.69, \\ 
   & 12.92) & 12.37) & 11.49) & 10.53) & 10.05) & 9.19) & 8.62)& 7.47) & 6.82) & 6.61) \\ 
   \hline 
 3.0 & 9.54 & 9.08 & 8.35 & 7.55 & 7.15 & 6.44 & 5.96 & 5.01 & 4.47 & 4.29 \\ 
   & (7.13, & (6.79, & (6.24, & (5.64, & (5.35, & (4.81, & (4.46, & (3.74, & (3.34, & (3.21, \\ 
   & 12.76) & 12.15) & 11.17) & 10.11) & 9.57) & 8.62) & 7.98) & 6.70) & 5.98) & 5.74) \\ 
   \hline 
  5.0 & 9.30 & 8.75 & 7.87 & 6.91 & 6.44 & 5.58 & 5.01 & 3.86 & 3.22 & 3.00 \\ 
  & (6.95, & (6.54,& (5.88, & (5.17, & (4.81, & (4.17, & (3.74, & (2.89, & (2.41, & (2.25, \\ 
  & 12.45) & 11.71) & 10.53) & 9.25) & 8.62) & 7.47) & 6.70) & 5.17) & 4.31) & 4.02) \\ 
  \hline 
  8.0& 9.17 & 8.56 & 7.60 & 6.56 & 6.04 & 5.10 & 4.47 & 3.22 & 2.51 & 2.28 \\ 
  & (6.85, & (6.40, & (5.68, & (4.90, & (4.51, & (3.81, & (3.34, & (2.41, & (1.88, & (1.70, \\ 
  & 12.27) & 11.46) & 10.17) & 8.78) & 8.08) & 6.82) & 5.98) & 4.31) & 3.37) & 3.05) \\ 
  \hline 
  10.0& 9.12 & 8.50 & 7.51 & 6.44 & 5.90 & 4.94 & 4.29 & 3.00 & 2.28 & 2.04 \\ 
   &(6.82, & (6.35, & (5.61, & (4.81, & (4.41, & (3.69, & (3.21, & (2.25, & (1.70, & (1.52, \\ 
   & 12.21) & 11.38) & 10.05) & 8.62) & 7.90) & 6.61) & 5.74) & 4.02) & 3.05) & 2.73) \\ 
   \hline
\end{tabular}
}
\end{table}

\clearpage

\section{Discussion}

A crucial task in causal inference with observational studies is to assess the sensitivity of causal conclusions with respect to unmeasured confounding. In sensitivity analysis, because one is assessing the sensitivity of conclusions to the assumption of no unmeasured confounding, additional untestable assumptions may often seem undesirable and suspect to researchers. We have introduced a new joint bounding factor that allows researchers to conduct sensitivity analysis without assumptions, i.e., we provide an inequality, that is applicable without any assumptions, such that the sensitivity analysis parameters must satisfy the inequality if an unmeasured confounder is to explain away the observed effect estimate or reduce it to a particular level. We can obtain a conservative estimate of the true causal effect by dividing the observed relative risk by the bounding factor; the method does not assume a single binary confounder or no exposure-confounder interaction on the outcome.

Previous sensitivity analysis approaches in the literature often relied on the assumption of a single binary confounder and no-interaction between the effects of the exposure and the confounder on the outcome \cite{Rosenbaum::1983JRSSB, Lin::1998, Vanderweele::2011}. For example,  Schlesselman \cite{Schlesselman::1978} assumed a binary confounder, a common relative risk, $\gamma$, of the confounder on the outcome for both with and without exposure, i.e., a no interaction assumption. Under these assumptions, he obtained the bias factor $ \RR_{ED}^\obs  / \SRR_{ED} =  \{ 1 + (\gamma-1) \pr(U=1\mid E=1) \} / \{ 1 + (\gamma - 1)\pr(U=1\mid E=0)  \}$ for sensitivity analysis requiring specifications of $\gamma, \pr(U=1\mid E=1)$ and $\pr(U=1\mid E=0).$ Our result requires fewer assumptions and fewer sensitivity parameters (two rather than three). We further discuss in the eAppendix that, under Schlesselman's formula, if $\pr(U=1\mid E=1)/P(U=1\mid E=0)$ is constrained to be no larger than some limit $\RR_{EU}$, then the maximum bias factor that can be obtained from Schlesselman's formula is $\RR_{EU}\times \gamma/(\RR_{EU}+\gamma-1)$, which is the same as our bounding factor. Thus, in this setting Schlesselman's no interaction assumption does not strengthen the bounds; the no interaction assumption is unnecessary. Without the no interaction assumption, Flanders and Khoury \citep{Flanders::1990} and VanderWeele and Arah \citep{Vanderweele::2011}, derived general formulas for sensitivity analysis. However, unless the confounder is binary, these formulas require specifying a very large number of parameters. They also require specifying the prevalence of each confounder level. Flanders and Khoury \citep{Flanders::1990} derive bounds for the true causal relative risk for the exposed population which are potentially applicable without specifying the prevalence of the unmeasured confounder. However, without specifying the prevalence, their formula  only leads to a low threshold Cornfield condition, and these bounds are thus much weaker than those in this paper. We discuss further the relation between their results and ours in the eAppendix.

The relative risk scale is widely used for sensitivity analysis in epidemiology and elsewhere, but the risk difference scale is also often of interest and importance \cite{Poole::2010, Ding::2014}. We show, in the Appendix, that similar conditions for sensitivity analysis also hold for the risk difference. If we use similar sensitivity parameters on the {\it relative risk scale} for the risk difference estimate, then we can derive similar lower bounds on the effects and determine how much confounding is required to explain away an effect or reduce it to a specific level. See Appendix 2 for details. SAS code for this approach is also given in the eAppendix. We can also do sensitivity analysis for the risk difference using sensitivity parameters on the {\it risk difference scale}. Unfortunately, however, these conditions for the risk difference using risk difference sensitivity parameters then depend on the number of categories of the unmeasured confounder, and become weaker for confounders with more categories. This is not the case for sensitivity analysis of the risk difference (or the relative risk) if the sensitivity parameters themselves are expressed on the relative risk scale, in which case the bounding factor is applicable and is the same regardless of the number of categories. Due to this property, it is perhaps more suitable to conduct sensitivity analysis for the risk difference using sensitivity parameters on the relative risk scale. See Appendix 3 for further discussion.

The hazard ratio is widely used for analyzing data with time-to-event outcome. In the eAppendix, we show that under the assumption of having a rare outcome at the end of follow-up, the same bounding factor also applies to the hazard ratio with the confounder-outcome relative risk replaced by the confounder-outcome hazard ratio. Likewise similar results also apply to non-negative outcomes (e.g., counts or positive continuous outcomes) by replacing the confounder-outcome relative risk by the maximum ratio by which the confounder may increase the expected outcome comparing any two confounder categories.

The new joint bounding factor $(\GRR_{EU}\times \GRR_{UD}) / ( \GRR_{EU} + \GRR_{UD} - 1)$ plays a central role in our sensitivity analysis approach, which, in turn, gives us a new measure of the strength of unmeasured confounding induced by a confounder $U.$ 
Our approach has the advantage of making no assumptions about the structure of the unmeasured confounder or confounders, and of delivering conclusions much stronger than the original Cornfield conditions.

In general, a table with many different possible sensitivity analysis parameters including values that are quite extreme, such as Table \ref{tb::table} above, will be most informative. However, at the very least, in any observational study, researchers should report how much confounding would be needed to reduce the estimate, and how much confounding would be needed to reduce the confidence interval, to include the null. We believe that if this were always done in observational studies, the evidence for causality could much more easily be assessed and science would be better served.

\section*{Appendix 1: SAS Code}

The SAS code for the cigarette smoking and lung cancer example in Table \ref{tb::table} is given below. A researcher could modify the code for use in other examples by just changing the first few lines of code with the estimated observed relative controlling for only the measured covarates (RR=), and the lower and upper confidence interval for this estimate(RR\_Lower=, RR\_Upper=). The minimum and maximum strength of the unmeasured confounder can also be modified by adjusting the lines with ``RR\_EU='' and ``RR\_UD='' but we recommend always including at least some relatively large values, e.g., with $\RR_{EU}$ and $\RR_{UD}$ at least as high as $5$ so as to get a sense as to how an estimate would change under fairly severe confounding.

{\setlength{\baselineskip}{0.5\baselineskip}
\scriptsize
\begin{verbatim}
proc iml;
/*the point estimator and confidence interval of RR*/
RR = 10.73;
RR_Lower = 8.02;
RR_Upper = 14.36;
/*strenghth of confounding resulting from U*/
RR_EU = {1.2 1.3 1.5 1.8 2 2.5 3 4 5 6 8 10};
RR_UD = {1.2 1.3 1.5 1.8 2 2.5 3 4 5 6 8 10};
highthreshold = ROUND(RR + SQRT(RR*(RR-1)), 0.01);
rownames_EU = CHAR(RR_EU, NCOL(RR_EU), 1);
colnames_UD = CHAR(RR_UD, NCOL(RR_UD), 1);
BiasFactor  = J(NCOL(RR_EU), NCOL(RR_UD), 1);
SPACE       = J(NCOL(RR_EU), NCOL(RR_UD), " ");
LeftP        = J(NCOL(RR_EU), NCOL(RR_UD), "(");
Mid         = J(NCOL(RR_EU), NCOL(RR_UD), ",");
RightP       = J(NCOL(RR_EU), NCOL(RR_UD), ")");
RR_true = BiasFactor;
RR_true_Lower = BiasFactor;
RR_true_Upper = BiasFactor;
RR_true_CI = BiasFactor;
DO i=1 TO NCOL(RR_EU);
      Do j=1 to NCOL(RR_UD);
	      BiasFactor[i, j] = RR_EU[i]*RR_UD[j]/(RR_EU[i] + RR_UD[j] - 1);
		  RR_true[i, j] = ROUND(RR/BiasFactor[i, j], 0.01);
          RR_true_Lower[i, j] = ROUND(RR_Lower/BiasFactor[i, j], 0.01);
          RR_true_Upper[i, j] = ROUND(RR_Upper/BiasFactor[i, j], 0.01);
	  END;
END;
RR_true_CI = CATX(" ", CHAR(RR_true), LeftP, CHAR(RR_true_Lower), Mid, CHAR(RR_true_Upper), RightP);
print RR_true_CI[colname = colnames_UD
                 rowname = rownames_EU
                 label = "Bounds on corrected estimates and confidence intervals for unmeasured confounding 
                 (columns correspond to increasing strength of the risk ratio of U on the outcome; 
                 rows correspond to increasing strength of risk ratio relating the exposure and U)"];
run;
\end{verbatim}
}

\section*{Appendix 2: Conditions for the Risk Difference Using Sensitivity Parameters on the Relative Risk Scale}

As in the text we assume analysis is conducted conditional on, or within strata of the measured covariates $C$.
Define the bounding factor as $\BF_U = \GRR_{EU}\times \GRR_{UD} / ( \GRR_{EU} + \GRR_{UD} - 1)$, the prevalence of the exposure as $f = \pr(E=1)$, and the probabilities of the outcome with and without exposure as $p_1=\pr(D=1\mid E=1) $ and $p_0=\pr(D=1\mid E=0).$ The causal risk differences for the exposed and unexposed populations are
\begin{eqnarray*}
\SRD_{ED+} &=& p_1 - \sum_{k=0}^{K-1} \pr(D=1\mid E=0,U=k)\pr(U=k\mid E=1),\\ 
\SRD_{ED-}  &=&   \sum_{k=0}^{K-1} \pr(D=1\mid E=1,U=k)\pr(U=k\mid E=0) - p_0,
\end{eqnarray*}
and the causal risk difference for the whole population is 
\begin{eqnarray*}
\SRD_{ED} &=&  \sum_{k=0}^{K-1} \{  \pr(D=1\mid E=1,U=k)  - \pr(D=1\mid E=0,U=k) \}  \pr(U=k) \\
&=&f\SRD_{ED+}  + (1-f) \SRD_{ED-}.
\end{eqnarray*} 
We show in the eAppendix that the lower bounds for the causal risk differences are
\begin{eqnarray*}
&&\SRD_{ED+}  \geq   p_1 - p_0\times \BF_U,\\
&&\SRD_{ED-} \geq  p_1/\BF_U - p_0, \\
&&\SRD_{ED} \geq   (     p_1    - p_0\times \BF_U )  \times  \left\{    f  + (1-f)/\BF_U  \right\}
=(    p_1 / \BF_U   - p_0)  \times  \left\{    f\times \BF_U + (1-f)  \right\} .
\end{eqnarray*}
Note that even without knowing $f$, we can use the inequality $ \SRD_{ED} \geq \min (  \SRD_{ED+} ,    \SRD_{ED-}   ) $ to obtain a lower bound for $ \SRD_{ED}$. 

As an example, suppose the probabilities of the outcome with and without exposure are $p_1=0.25, p_0=0.1$, and therefore the observed risk difference is $\RD_{ED}^\obs  = p_1-p_0=0.15$. If we assume that the unmeasured confounding measures are $(\RR_{EU}, \RR_{UD}) = (2,2)$ with the joint bounding factor of $2\times 2/(2+2-1)=1.33$, then the true risk difference for the exposed is at least as large as $0.25-0.1\times 1.33 = 0.12$, the true risk difference for the unexposed is at least as large as $0.25/1.33-0.1=0.09$, and the true risk difference for the whole population is at least as large as $\min(0.12,0.09)=0.09$. If we further know that the prevalence of the exposure is $f=0.2$, the true risk difference for the whole population is at least as large as $  0.12\times 0.2+0.09\times 0.8 =0.10.$

The above results imply that, for an unmeasured confounder to reduce the observed risk difference to be $\SRD_{ED+}, \SRD_{ED-}$ and $\SRD_{ED}$ respectively, the Cornfield conditions for the joint bounding factor for the exposed, the unexposed, and the whole population, respectively, are
\begin{eqnarray*}
&&\BF_U \geq   (p_1 - \SRD_{ED+} ) / p_0 ,\\
&&\BF_U \geq   p_1 / (p_0+ \SRD_{ED-}) ,\\
&&\BF_U \geq     \frac{   \sqrt{   \{  \SRD_{ED} + p_0(1-f) -  p_1f   \}^2+4p_1p_0f(1-f)       }  -    \{  \SRD_{ED} + p_0(1-f) - p_1f   \} }{   2p_0f }.
\end{eqnarray*}
Note that if the true causal risk difference is $\SRD_{ED}=0$, the above conditions all reduce to $\BF_U\geq \RR_{ED}^\obs .$ 
Suppose, again, the probabilities of the observed outcome with and without exposure are $p_1=0.25, p_0=0.1$, and the prevalence of the exposure is $f=0.2$. For an unmeasured confounder $U$ to reduce the observed risk difference of $\RD_{ED}^\obs =0.15$ to a true risk difference of $\SRD_{ED} = 0.05$, the joint bounding factor resulting from the confounder must be at least as large as
\begin{eqnarray*}
&\BF_U \geq  \frac{   \sqrt{   (0.05+0.1\times 0.8 - 0.25\times 0.2)^2 +   4\times 0.25\times 0.1\times 0.2\times 0.8 }   -   (0.05+0.1\times 0.8 - 0.25\times 0.2)     }{  2\times 0.1\times 0.2  } = 1.74.&
\end{eqnarray*}
Therefore, as in the text both of the confounding measures $\RR_{EU}$ and $\RR_{UD}$ must be at least as large as $1.74$, and the maximum of them must be at least as large as $1.74+\sqrt{1.74\times 0.74} = 2.88.$

The above results are useful for apparently causative exposures with $\RD_{ED}^\obs  > 0$, which give (possibly positive) lower bounds for the causal risk differences. However, for apparently preventive exposure with $\RD_{ED}^\obs  < 0$, we need to modify the definition of $\GRR_{EU}$ as $\GRR_{EU} = \max_u \GRR^{-1}_{EU}(u)$. And we have the following analogous results on the upper bounds of the causal risk differences:
\begin{eqnarray*}
&&\SRD_{ E D+}  \leq   p_1\times \BF_U - p_0 ,\\
&&\SRD_{ E D-} \leq  p_1  - p_0/\BF_U, \\
&&\SRD_{ E D} \leq  (   p_1\times \BF_U - p_0)  \times  \left\{    f  + (1-f)/\BF_U  \right\}
=(    p_1  - p_0/\BF_U  ) \times  \left\{    f\times \BF_U + (1-f)  \right\} .
\end{eqnarray*}

Due to the linearity of the risk difference, we can also obtain the lower bound of the marginal risk differences averaged over the observed covariates $C$ using 
$\SRD_{ED+} = \sum_c \SRD_{ED|c+} \pr(C=c\mid E=1), \SRD_{ED-} = \sum_c \SRD_{ED|c-} \pr(C=c\mid E=0)$ and $\SRD_{ED} = \sum_c \SRD_{ED|c} \pr(C=c).$
In the eAppendix, we provide details and proofs for the results above, discuss statistical inference for the causal risk difference bounds under finite samples, and give formulas for how large the bounding factor would have to be to reduce an estimate or a confidence interval to $0$ or to some other specified quantity. In the eAppendix, we also provide software code to implement this sensitivity analysis approach for the risk difference.

\section*{Appendix 3: Conditions for the Risk Difference Using Sensitivity Parameters on the Risk Difference Scale}

In the previous Appendix, we considered sensitivity analysis for the risk difference with sensitivity analysis parameters on the relative risk scale. In this Appendix, we consider sensitivity analysis for the risk difference with parameters defined on the risk difference scale. Unfortunately, for the reasons described below, the results for the risk difference with parameters defined on the difference scale are not as practically useful as when the parameters are defined on the relative risk scale.

Let $\RD_{ED}^\obs  = \pr(D=1\mid E=1) -  \pr(D=1\mid E=0)$ denote the observed risk difference, and
\begin{eqnarray*}
\SRD_{ED} = \sum_{k=0}^{K-1} \{ \pr(D=1\mid E=1, U=k) - \pr(D=1\mid E=0, U=k)\} \pr(U=k)
\end{eqnarray*}
 denote the standardized risk difference.

Define $\alpha_k = \pr(U=k\mid E=1) - \pr(U=k\mid E=0)  $ as the difference in the probability that the confounder $U$ takes a particular value $k$ comparing exposed and unexposed. We use $\GRD_{EU} = \max_{k\geq 1}  |\alpha_k|$, the maximum of these absolute differences, to measure the exposure-confounder association on the risk difference scale, defined as the maximal risk difference of the exposure $E$ on the confounder $U$. Define $\beta_{k}^* = \pr(D=1\mid E=1, U=k) - \pr(D=1\mid E=1, U=0) $ and $\beta_{k} = \pr(D=1\mid E=0, U=k) - \pr(D=1\mid E=0, U=0)  $ as the difference in the probability of the outcome comparing the category $k$ and $0$ of the confounder $U$ with and without exposure. Define $\GRD_{UD|E=1}=\max_{k\geq 1}  |\beta_{k}^*|$ and $\GRD_{UD|E=0} = \max_{k\geq 1}  |\beta_{k}|$ as the maximums of these differences with and without exposure, respectively. We use $\GRD_{UD} = \max(\GRD_{UD|E=1}, \GRD_{UD|E=0})$ to measure the confounder-outcome association in the risk difference scale, defined as the maximal risk difference of the confounder $U$ on the outcome $D$.

We first consider a {\it binary} unmeasured confounder.
For binary confounder $U$, the maximal risk difference $\GRD_{EU} $ becomes the ordinary risk difference $\RD_{EU}$, and the maximal risk difference becomes the maximum of two conditional risk difference $\GRD_{UD} = \max(\RD_{UD|E=1}, \RD_{UD|E=0})$. We have that
$$
\RD_{EU} \times \GRD_{UD} \geq  \RD_{ED}^\obs  - \SRD_{ED},
$$
which further leads to the following low and high thresholds:
$$
\min(    \RD_{EU},  \GRD_{UD}   ) \geq \RD_{ED}^\obs  -  \SRD_{ED} ,
\qquad
\max(   \RD_{EU},  \GRD_{UD}   ) \geq \sqrt{  \RD_{ED}^\obs  -  \SRD_{ED} }  ,
$$
which generalize previous results under the null of zero causal effect of $E$ on $D$ \cite{Gastwirth::1998, Poole::2010, Ding::2014}.

For categorical confounder $U$, no simple form of the bounding factor is available, but we can still show that $\GRD_{EU}$ and $\GRD_{UD} $ must satisfy the following conditions:
\begin{eqnarray*}
\GRD_{EU} &\geq& (\RD_{ED}^\obs  -  \SRD_{ED} ) /(K-1),  \\ 
\GRD_{UD} &\geq& (\RD_{ED}^\obs  -  \SRD_{ED} )   /2,  \\
 \max(\GRD_{EU},  \GRD_{UD}   ) &\geq& \max\left\{  \sqrt{  (\RD_{ED}^\obs  -  \SRD_{ED} )   /(K-1) } , (\RD_{ED}^\obs  -  \SRD_{ED} ) /2    \right\} .
\end{eqnarray*}
When $K=3$ such as a three-level genetic confounder, these conditions reduce to
\begin{eqnarray}
\min(\GRD_{EU} , \GRD_{UD} ) \geq  (\RD_{ED}^\obs  -  \SRD_{ED} )  /2,\quad 
\max(\GRD_{EU} , \GRD_{UD} ) \geq \sqrt{ (\RD_{ED}^\obs  -  \SRD_{ED} )  /2}.
\end{eqnarray}

The results above generalize previous results \citep{Ding::2014} from the null hypothesis of no effect ($\SRD_{ED} = 0$) to alternative hypotheses ($\SRD_{ED}$ arbitrary). We show the proofs and extensions for the above results in the eAppendix.

We can see from above that the generalized Cornfield conditions for the risk difference under alternative hypotheses depend on the number of categories of $U$, and become less informative as the number of categories increases. 
Therefore, a binary confounder is not the most conservative case for sensitivity analysis with parameters expressed the risk difference scale.
However, the Cornfield conditions for the relative risk do not suffer from this problem. Therefore, it seems that it is more appropriate to conduct sensitivity analysis with parameters expressed on the risk ratio scale, and a binary confounder is the most conservative case for sensitivity analysis with parameters expressed on the risk ratio scale \citep{Wang::2006, Ichino::2008}.


\newpage

\begin{center}
\bf\Huge 
Supplementary Materials for ``Sensitivity Analysis Without Assumptions''
\end{center}

 \setcounter{section}{0}
 \setcounter{equation}{0}
 \setcounter{proposition}{0}
 \setcounter{lemma}{0}
 \setcounter{figure}{0}

\renewcommand{\theequation}{A.\arabic{equation}}
\renewcommand{\thelemma}{A.\arabic{lemma}}
\renewcommand{\theproposition}{A.\arabic{proposition}}
\renewcommand{\thesection}{Appendix~\arabic{section}}
\renewcommand{\thefigure}{A.\arabic{figure}}

The eAppendix contains the following nine sections:
\begin{enumerate}
[{Appendix}~1{:}]
\item
Three useful lemmas which are used repeatedly in the proofs in later sections;

\item 
The new bounding factor introduced in the main text and its implied Cornfield conditions with proofs;

\item
Another bounding factor with the exposure-confounder relationship on the odds ratio scale and its implied Cornfield conditions with proofs;

\item
Relations between the new bounding factor and some existing results including Schlesselman's formula \citep{Schlesselman::1978} and Flanders and Khoury's results \citep{Flanders::1990}; 

\item
Results for the risk difference using sensitivity parameters on the relative risk scale with proofs;

\item
SAS code for the risk difference using sensitivity parameters on the relative risk scale;

\item
Results for the risk difference using sensitivity parameters on the risk difference scale with proofs;

\item
A bounding factor for rare time-to-event outcome on the hazard ratio scale and its implied Cornfield conditions;

\item
A bounding factor for general nonnegative outcomes.

\end{enumerate}

\section{Useful Lemmas}
\begin{lemma}
\label{lemma::basic-factorial-derivative}
Define 
$
h(x) = (c_1 x  +1) / (c_2 x + 1). 
$
When $c_1  >  c_2$, $h'(x) > 0$ and $h(x)$ is increasing; when $c_1 \leq c_2$, $h'(x) \leq 0$ and $h(x)$ is non-increasing.
\end{lemma}

\begin{proof}
[Proof of Lemma \ref{lemma::basic-factorial-derivative}]
The first derivative of $h(x)$ is
$$
h'(x) = \frac{  c_1(c_2 x + 1) - (c_1 x + 1) c_2  }{  (c_2 x + 1)^2 } = \frac{c_1 - c_2}{ (c_2 x+ 1)^2}.
$$
When $c_1 >c_2$, $h'(x) > 0$ and $h(x)$ is increasing in $x$. When $c_1 \leq c_2$, we have opposite results. 
\end{proof}

\begin{lemma}
\label{lemma::cornfield-increasing}
When $x,y > 1$, $h(x,y) =(xy)/(x+y-1)$ is increasing in both $x$ and $y$.
\end{lemma}

\begin{proof}
[Proof of Lemma \ref{lemma::cornfield-increasing}]
The first partial derivative of $h(x,y)$ with respect to $x$ is
$$
\frac{\partial  h(x,y)}{ \partial x} = \frac{   y(x+y-1) - xy  }{  (x+y-1)^2 } = \frac{y(y-1)}{  (x+y-1)^2 }.
$$
When $x,y>1$, $\partial h(x,y)/ \partial x >0$ and $h(x,y)$ is increasing in $x$. By symmetry, the conclusion holds also for $y.$
\end{proof}

\begin{lemma}
\label{lemma::increasing}
When $x,y>1$, $h(x,y) = ( \sqrt{xy}  + 1  )/(  \sqrt{x}  + \sqrt{y} ) $ is increasing in both $x$ and $y$.
\end{lemma}

\begin{proof}
[Proof of Lemma \ref{lemma::increasing}]
The first partial derivative of $h(x,y)$ with respect to $x$ is
$$
\frac{\partial h(x,y) }{ \partial x} = \frac{  \frac{1}{2}\sqrt{y/x} (  \sqrt{x} + \sqrt{y}) -   \frac{1}{2}  ( \sqrt{xy}  + 1  ) /\sqrt{x}    }{   (  \sqrt{x} + \sqrt{y})^2   }
=  \frac{ y-1    }{  2\sqrt{x} (  \sqrt{x} + \sqrt{y})^2   }.
$$
When $x,y>1$, $\partial h(x,y)/ \partial x >0$ and $h(x,y)$ is increasing in $x$. By symmetry, the conclusion holds also for $y.$
\end{proof}

\section{The New Bounding Factor and Implied Cornfield Conditions}
\label{sec::cornfield-alternative}

\subsection{Technical Measure-Theoretical Details}

This subsection presents the technical framework for the proofs. A less technical reader can skip this subsection and move directly to the next subsection \ref{sec::bias-factor} on the new bounding factor. Throughout the eAppendix, we allow the unmeasured confounder $U$ to take arbitrary values, which is a measurable mapping from probability space $(\Omega, \mathcal{F}, \pr )$ to a measurable space $( \Upsilon, \mathcal{U})$. For $\V \in \mathcal{U}$, we define $F_1( \V ) = \pr(U\in \V \mid E=1)$ as the distribution of $U$ with exposure, $F_0(\V) = \pr(U\in \V \mid E=0)$ as the distribution of $U$ without exposure, and $F(\V) = \pr(U\in \V)$ as the marginal distribution of $U$. The distributions $F_1(\cdot), F_0(\cdot)$ and $F(\cdot)$ are measurable mappings from $\Upsilon$ to $[0, 1]$, which correspondingly induce three probability measures on the measurable space $( \Upsilon, \mathcal{U})$. When the confounder $U$ is a scalar on the real line, these definitions reduce to $F_1( u ) = \pr(U\leq u\mid E=1)$, the cumulative distribution function (CDF) of $U$ with exposure, $F_0(u) = \pr(U\leq u\mid E=0)$, the CDF of $U$ without exposure, and $F(u) = \pr(U\leq u)$, its marginal CDF. Correspondingly, the CDFs, $F_1, F_0,$ and $ F$, also induce three measures on the real line. In the following, we assume that the measure $F_1$ is absolutely continuous with respect to the measure $F_0$, with the Radon--Nikodym derivative defined as $\GRR_{EU}(u) = F_1( du) / F_0( du)$, which is the generalized relative risk of $E$ on $U$ at $U=u$. The absolute continuous assumption about $F_1$ and $F_0$ holds automatically for categorical and absolutely continuous unmeasured confounder $U$. For general confounder $U$, this is only a mild regularity condition.

\subsection{The New Bounding Factor}\label{sec::bias-factor}

We assume for the next several sections that analysis is done conditional on, or within strata of the measured confounders $C.$
We define the maximal relative risk of $E$ on $U$ as $\GRR_{EU} = \max_u \GRR_{EU}(u)$. Define $r(u) = \pr(D=1\mid E=0, U=u)$ and $r^*(u) = \pr(D=1\mid E=1, U=u)$ as the probabilities of the outcome within stratum $U=u$ without and with exposure. Define the maximal relative risk of $U$ on $D$ as $\GRR_{UD|E=0} =  \max_u r(u) / \min_u r(u)$ and $ \GRR_{UD|E=1} = \max_u r^*(u)/ \min_u r^*(u)$ without and with exposure, and $\GRR_{UD} = \max(  \GRR_{UD|E=0}, \GRR_{UD|E=1} )$ as the maximum of these two relative risks. The maxima and minima are taken over the space $\Upsilon$, and hereinafter. When $U$ is a categorical confounder with levels $0,1,\ldots, K-1$, the definitions above reduce to the definitions in the main text. 
To allow for causal interpretations, we invoke the counterfactural or potential outcomes framework, with $D_i(1)$ and $D_i(0)$ being the potential outcomes for individual $i$ with and without the exposure, respectively; we also need to
make the ignorability assumption \citep{Rosenbaum::1983} $E\ind \{ D(1), D(0) \} \mid U$.

The observed relative risk of the exposure $E$ on the outcome $D$ is
\begin{eqnarray*}
\RR_{ED}^\obs  =   \frac{  \int   \pr(D=1\mid E=1, U=u ) F_1(du)  }{ \int  \pr(D=1\mid E=0, U=u)  F_0(du) }  =  \frac{  \int   r^*(u)  F_1(du)  }{ \int    r(u) F_0(du) },
\end{eqnarray*}
where the integrals are over $\Upsilon$ and hereinafter.
The relative risks standardized by the exposed, the unexposed, and the whole population are as follows:
\begin{eqnarray*}
\SRR_{ED+}  =  \frac{  \int    r^*(u) F_1(du) }{  \int   r(u)  F_1(du) },
\quad
\SRR_{ED-}  =  \frac{  \int    r^*(u) F_0(du)  }{  \int    r(u) F_0(du) },
\quad
\SRR_{ED}   =  \frac{  \int  r^*(u) F(du)  }{  \int   r(u) F(du) }.
\end{eqnarray*}
When unmeasure confounder $U$ is categorical, $\SRR_{ED}$ reduces to the form in the main text, and all other relative risk measures can be simplifies by replacing integrations by summations.
The corresponding confounding relative risks standardized by the exposed, the unexposed, and the whole population are
$$
\CRR_{ED+} = { \RR_{ED}^\obs \over  \SRR_{ED+}} =   \frac{  \int   r(u)  F_1(du)  }{  \int    r(u) F_0(du) }  ,
\quad 
\CRR_{ED-}  = {\RR_{ED}^\obs  \over \SRR_{ED-}}  =  \frac{  \int    r^*(u) F_1(du) }{  \int    r^*(u) F_0(du) } ,
$$
and $\CRR_{ED}     = \RR_{ED}^\obs  /  \SRR_{ED}  $.
Similar to Lee \cite{Lee::2011}, we have that $\SRR_{ED}$ is a weighted average of $\SRR_{ED+}$ and $\SRR_{ED-}$, and $\CRR_{ED}$ is a harmonic average of $\CRR_{ED+}$ and $\CRR_{ED-}$.

\begin{proposition}
\label{prop::average}
We have 
$$
\SRR_{ED} = w \SRR_{ED+} + (1-w) \SRR_{ED-},\quad 
1/\CRR_{ED} = w/ \CRR_{ED+} + (1-w)/ \CRR_{ED-}  ,
$$
where $f=\pr(E=1)$ and $w$ is a weight between zero and one:
$$
w = \frac{ f   \int   r(u) F_1(du)    }{  f  \int   r(u) F_1(du)  + (1-f)  \int   r(u) F_0(du) } \in [0,1] .
$$
\end{proposition}

\begin{proof}[Proof of Proposition \ref{prop::average}.]
The conclusions follow from the following decomposition:
\begin{eqnarray*}
\SRR_{ED} &=&  \frac{  \int  r^*(u) F(du)  }{  \int   r(u) F(du) } 
= \frac{   f \int  r^*(u) F_1(du) + (1-f)\int  r^*(u) F_0(du)   }{  f \int   r(u) F_1(du) + (1-f)  \int   r(u) F_0(du) }  \\
&=&        \frac{ f   \int   r(u) F_1(du)    }{   f  \int   r(u) F_1(du)  + (1-f)  \int   r(u) F_0(du) }  
              \times   \frac{   \int  r^*(u) F_1(du)   }{    \int   r(u) F_0(du)  }   \\
&&              +
              \frac{ (1-f)  \int   r(u) F_0(du)   }{   f  \int   r(u) F_1(du)  + (1-f)  \int   r(u) F_0(du) }  
              \times  \frac{     \int   r^*(u) F_0(du)     }{   \int   r(u) F_0(du)        }.
\end{eqnarray*}
\end{proof}

The confounding relative risks can be bounded from above by the bounding factor 
$$
\BF_U =  { \GRR_{EU}\times \GRR_{UD} \over  \GRR_{EU} +  \GRR_{UD}  - 1 },
$$ 
as shown in the following proposition.

\begin{proposition}
\label{prop::bound-CRR}
The confounding relative risks can be bounded from above by
\begin{eqnarray*} 
\CRR_{ED+} = \frac{\RR_{ED}^\obs}{\SRR_{ED+}} \leq   \BF_U, \quad 
\CRR_{ED-} = \frac{\RR_{ED}^\obs}{\SRR_{ED-}} \leq \BF_U,\quad 
\CRR_{ED} = \frac{\RR_{ED}^\obs}{\SRR_{ED}} \leq  \BF_U.
\end{eqnarray*}
\end{proposition}

\begin{proof}[Proof of Proposition \ref{prop::bound-CRR}.]
In the following proof, we first discuss $\CRR_{ED+}.$ The key observation is to write $\CRR_{ED+}$ in terms of a binary confounder with two levels corresponding to $\max_u r(u)$ and $\min_u r(u)$. To be more specific, we have that 
\begin{eqnarray*}
\CRR_{ED+} = \frac{   w_1 \max_u r(u)  + (1 - w_1)  \min_u r(u) }{  w_0 \max_u r(u)  +(1 - w_0) \min_u r(u)    } ,
\end{eqnarray*}
where 
\begin{eqnarray*}
w_1 =   \frac{  \int   \{ r(u)  - \min_u r(u) \}  F_1(du)  }{  \max_u r(u) - \min_u r(u) }, &&
1-w_1 =  \frac{  \int   \{ \max_u r(u) - r(u) \}  F_1(du)  }{  \max_u r(u) - \min_u r(u) } ,\\
w_0 =  \frac{  \int \{ r(u)  - \min_u r(u)\}   F_0(du)  }{  \max_u r(u) - \min_u r(u) }, &&
1-w_0 =  \frac{  \int \{ \max_u r(u) - r(u)\}  F_0(du)   }{  \max_u r(u) - \min_u r(u) } .
\end{eqnarray*}
Define $\Gamma = w_1/w_0$, and we have
\begin{eqnarray*}
\Gamma &=&\frac{ w_1}{w_0} = 
\frac{  \int   \{ r(u)  - \min_u r(u) \}  F_1(du) }{  \int \{ r(u)  - \min_u r(u)\}   F_0(du)  }
= \frac{  \int \{ r(u)  - \min_u r(u)\}  \GRR_{EU}(u)  F_0(du)  }{  \int \{ r(u)  - \min_u r(u)\}   F_0(du) } \\
&\leq& \frac{  \max_u \GRR_{EU}(u) \times  \int \{ r(u)  - \min_u r(u)\}   F_0(du)  }{  \int \{ r(u)  - \min_u r(u)\}   F_0(du) }
=\GRR_{EU}.
\end{eqnarray*} 
We can write $w_0 = w_1 / \Gamma$, and therefore we have
$$
\CRR_{ED}^+ =  \frac{  \{ \max_u r(u)  - \min_u r(u) \}  \times w_1 + \min_u r(u)   } {  \{ \max_u r(u)  - \min_u r(u) \} /\Gamma  \times w_1 + \min_u r(u)        }.
$$
In the following, we divide our discussion into two cases. If $\Gamma >  1$, then $\CRR_{ED}^+$ is increasing in $w_1$ according to Lemma \ref{lemma::basic-factorial-derivative}, and it attains the maximum at $w_1=1$. Thus we have
$$
\CRR_{ED}^+ \leq  \frac{   \Gamma \times \GRR_{UD|E=0}  }{  \Gamma +  \GRR_{UD|E=0}  - 1 }
\leq   \frac{   \GRR_{EU} \times \GRR_{UD|E=0}  }{ \GRR_{EU} +  \GRR_{UD|E=0}  - 1 },
$$
where the second inequality follows from Lemma \ref{lemma::cornfield-increasing}.
If $\Gamma \leq 1$, then $\CRR_{ED}^+$ is non-increasing in $w_1$, and it attains the maximum at $w_1=0$. Thus we have
$$
\CRR_{ED}^+ \leq  1 \leq  \frac{   \GRR_{EU} \times \GRR_{UD|E=0}  }{ \GRR_{EU} +  \GRR_{UD|E=0}  - 1 },
$$
where the the second inequality again follows from Lemma \ref{lemma::cornfield-increasing}.

The same discussion applies to $\CRR_{ED}^-$, and we can obtain that 
\begin{eqnarray*}
\CRR_{ED}^-   \leq   \frac{  \GRR_{EU}\times  \GRR_{UD|E=1}    }{  \GRR_{EU} +  \GRR_{UD|E=1}     - 1}.
\end{eqnarray*}
Using the fact $1 / \CRR_{ED} = w / \CRR_{ED}^+ +(1-w) /  \CRR_{ED}^- $, we know that
\begin{eqnarray*}
\frac{1 }{ \CRR_{ED}  }  \geq  \left(   \frac{     \GRR_{EU} \times  \GRR_{UD}    }{  \GRR_{EU} +  \GRR_{UD}    - 1} \right)^{-1},
\end{eqnarray*}
and the conclusion follows.
\end{proof}

\subsection{The Implied Cornfield Conditions}

Proposition \ref{prop::bound-CRR} says that the bounding factor is larger than or equal to all the confounding relative risks. It can be viewed as the Cornfield condition for the joint value of $(\GRR_{EU}, \GRR_{UD})$ in order to reduce the observed relative risk of $\RR_{ED}^\obs$ to the causal relative risk of $\SRR_{ED}$. If we specify one of the unmeasured confounding measure, for example $\GRR_{EU}$, then we can solve \ref{prop::bound-CRR} and obtain the lower bound of the other confounding measure:
$$
 \GRR_{UD} \geq   \frac{\GRR_{EU}\times  \RR_{ED}^\obs  -  \RR_{ED}^\obs  }{ \GRR_{EU} \times \SRR_{ED} - \RR_{ED}^\obs} .
$$ 
When $ \SRR_{ED}=1$, the above lower bound reduces to
$$
 \GRR_{UD} \geq   \frac{\GRR_{EU}\times  \RR_{ED}^\obs  -  \RR_{ED}^\obs  }{ \GRR_{EU}  - \RR_{ED}^\obs} .
$$
Further,  Proposition \ref{prop::bound-CRR}  implies the following Cornfield-type conditions for $\GRR_{EU}$ and $ \GRR_{UD}$.

\begin{proposition}
\label{prop::cornfield-conditions}
We have the following Cornfield conditions:
$$
\min(\GRR_{EU},  \GRR_{UD}) \geq \CRR_{ED}  ,\quad 
\max(\GRR_{EU},  \GRR_{UD}) \geq  \CRR_{ED} + \sqrt{  \CRR_{ED} ( \CRR_{ED} - 1)  } .
$$
\end{proposition}

\begin{proof}[Proof of Proposition \ref{prop::cornfield-conditions}.]
According to Lemma \ref{lemma::cornfield-increasing}, the right-hand side of the last inequality in Proposition \ref{prop::bound-CRR} is increasing in both $\GRR_{UD}$ and $\GRR_{EU}$. Therefore, the right-hand side of the above inequality in Proposition \ref{prop::bound-CRR} will increase if we let $\GRR_{UD}$ or $\GRR_{EU}$ go to large extremes. 
Let $\GRR_{UD}\rightarrow \infty$, and we have 
$
\CRR_{ED} \leq \GRR_{EU}. 
$
Let $\GRR_{EU}\rightarrow \infty$, and we have 
$
\CRR_{ED} \leq \GRR_{UD}. 
$
Therefore, we have the following low threshold:
$
\min(\GRR_{UD}, \GRR_{EU}) \geq \CRR_{ED} .
$
We can obtain the following inequality by replacing $\GRR_{UD}$ and $ \GRR_{EU}$ in the bounding factor by their maximum value due to Lemma \ref{lemma::cornfield-increasing}:
$$
\CRR_{ED} \leq  \frac{   \max^2(\GRR_{UD}, \GRR_{EU})  }{    2\max(\GRR_{UD}, \GRR_{EU}) - 1 },
$$
solving $\max(\GRR_{UD}, \GRR_{EU}) $ from which we can obtain the following high threshold.
\end{proof}

\subsection{Preventive Exposures}

The bounding factor in Proposition \ref{prop::bound-CRR} is particularly useful for an apparently causative exposure with $\RR_{ED}^\obs>1$, and the true causal relative risk is an attenuation of $\RR_{ED}^\obs$ by the bounding factor. However, for apparently preventive exposure with $\RR_{ED}^\obs < 1$, we can derive equally useful bias formula. For apparently preventive exposure, we modify the definition of the relative risk between $E$ and $U$ as $\GRR_{EU} = \max_u  \GRR_{EU}^{-1} (u) = 1/\min_u \GRR_{EU}(u)$, and obtain the following analogous result.

\begin{proposition}
\label{prop::preventive}
For apparently preventive exposure, we have
$
 \SRR_{ED} /   \RR_{ED}^\obs   \leq \BF_U.
$
Or, equivalently, the true causal relative risk is an inflation of $\RR_{ED}^\obs$ by the bounding factor.
\end{proposition}

\begin{proof}
[Proof of Proposition \ref{prop::preventive}.]
Define $\bar{E} = 1-E$, and the exposure $\bar{E}$ is apparently preventive for the outcome. Therefore, Proposition \ref{prop::bound-CRR} implies that
$$
\frac{  \RR_{\bar{E}D} }{  \SRR_{\bar{E}D}}  \leq  \frac{ \GRR_{\bar{E}U}\times \GRR_{UD} }{  \GRR_{\bar{E}U} +  \GRR_{UD}  - 1 }.
$$
Since $\RR_{\bar{E}D} = 1/\RR_{ED}^\obs , \SRR_{\bar{E}D} = 1 / \SRR_{ED}$, and $ \GRR_{\bar{E} U} = \max_u  \GRR_{EU}^{-1}(u)  = \GRR_{EU}$, the conclusion follows.
\end{proof}

\subsection{Averaged Over Observed Covariates}

All the results above are within strata of observed covariates $C$. The probabilities are conditional probabilities (e.g., $\pr(D=1\mid E=1, C=c), \pr\{  D(1) = 1\mid E=0, C=c \}, etc.$), the causal relative risks are conditional causal measures (e.g., $\SRD_{ED+} = \pr\{  D(1) = 1\mid E=1,  C=c \} / \pr\{  D(0) = 1\mid E=0, C=c \}, etc.$), and the bounding factor is also conditional denoted as $\BF_{U|c} = \GRR_{EU|c}\times \GRR_{UD|c} / ( \GRR_{EU|c} + \GRR_{UD|c} - 1 )$.

We have the following decomposition: 
\begin{eqnarray*}
\SRR_{ED} &=& \frac{     \int \pr(D=1\mid E=1, C=c, U=u)   F_{CU} (dcdu)      }{      \int \pr(D=1\mid E=0, C=c,  U=u)  F_{CU}(dcdu)   } \\
&=& \frac{  \int \int \pr(D=1\mid E=1, C=c, U=u) F_{U|C}(du)  F_C(dc)    }{    \int \int \pr(D=1\mid E=0, C=c, U=u) F_{U|C}(du)  F_C(dc)        } \\
&=& \frac{  \int  \pr\{  D(1)=1\mid C=c \}  F_C(dc)   }{  \int \pr\{  D(0)=1\mid C=c \}   F_C(dc)     } \\
&=& \frac{  \int  \SRR_{ED|c}   \pr\{  D(0)=1\mid C=c \}  F_C(dc)   }{  \int \pr\{  D(0)=1\mid C=c \}   F_C(dc)     } .
\end{eqnarray*}

Applying the result about conditional causal relative risk, we have
$$
\SRR_{ED} \geq 
\frac{  \int  \frac{  \RR_{ED|c}^\obs  }{ \BF_{U|c}} \pr\{  D(0)=1\mid C=c \}  F_C(dc)   }{  \int \pr\{  D(0)=1\mid C=c \}   F_C(dc)     }
 \geq \min_c \frac{  \RR_{ED|c}^\obs  }{ \BF_{U|c}}.
$$
If we assume a common causal relative risk $\SRR_{ED|c} = \SRR_{ED}$, then we can sharpen the result as:
$$
\SRR_{ED} \geq \max_c \frac{  \RR_{ED|c}^\obs  }{ \BF_{U|c}}.
$$

\section{Another Bounding Factor and Implied Cornfield Conditions Using the Odds Ratio}
\label{sec::another-bounds}

\subsection{Another Bounding Factor Using the Odds Ratio}
\label{sec::another}

Define $p(u) = \pr(E=1\mid U=u)$ as the probability of the exposure, and $q(u) = p(u)/\{ 1-p(u) \}$ as the odds of the exposure within level $u$ of the confounder $U$. Let $\GOR_{EU} = \max_u q(u) / \min_u q(u) $ be the ratio of the maximum and minimum of these odds. We use $\GOR_{EU} $ to measure the association between the confounder $U$ and the exposure $E$, which is defined as the maximal odds ratio between the exposure $E$ and the confounder $U$. When the confounder $U$ is binary, it reduces to the ordinary odds ratio. Using the odds ratio between the exposure $E$ and $U$ and the relative risk of the confounder $U$ on the outcome $D$ as the association measure as Bross and Lee \citep{Bross::1966, Bross::1967, Lee::2011}, we have the following bounding factor that ties $\CRR_{ED}$ with $\GOR_{EU} $ and $\GRR_{UD}$:

\begin{proposition}
\label{proposition::another}
We have
\begin{eqnarray}
\left(    \frac{   \sqrt{  \GOR_{EU} \GRR_{UD}  } + 1   }{   \sqrt{\GOR_{EU}}  + \sqrt{\GRR_{UD}} }  \right)^2 \geq  \frac{\RR_{ED}^\obs}{\SRR_{ED}} = \CRR_{ED}. 
\label{eq::bross-lee-whole}
\end{eqnarray}
\end{proposition}

\begin{proof}[Proof of Proposition \ref{proposition::another}.]
Lee \cite{Lee::2011} obtained the following results:
\begin{eqnarray}
\CRR^{+}_{ED} \leq   \left(  \frac{  \sqrt{ \GOR_{EU} \GRR_{UD|E=0}} + 1}{ \sqrt{\GOR_{EU}}  + \sqrt{\GRR_{UD|E=0}} } \right)^2, 
\quad 
\CRR^{-}_{ED}  \leq    \left(   \frac{  \sqrt{\GOR_{EU} \GRR_{UD|E=1}} + 1}{ \sqrt{\GOR_{EU}}  + \sqrt{\GRR_{UD|E=1}} } \right)^2 ,
\label{eq::bross-lee-sub}
\end{eqnarray}
Since $\GRR_{UD} = \max(\GRR_{UD|E=0}, \GRR_{UD|E=1})$, Lemma \ref{lemma::increasing} implies that
$$
\CRR^{+}_{ED} \leq  \left(  \frac{  \sqrt{ \GOR_{EU} \GRR_{UD}} + 1}{ \sqrt{\GOR_{EU}}  + \sqrt{\GRR_{UD}} } \right)^2,
\quad 
\CRR^{-}_{ED} \leq   \left(   \frac{  \sqrt{\GOR_{EU} \GRR_{UD}} + 1}{ \sqrt{\GOR_{EU}}  + \sqrt{\GRR_{UD}} } \right)^2 ,
$$
which leads to
$$
\frac{1}{\CRR_{ED}}   = \frac{w}{\CRR^{+}_{ED}} + \frac{1-w}{\CRR^{-}_{ED}} \geq   
 \left(   \frac{ \sqrt{\GOR_{EU}}  + \sqrt{\GRR_{UD}} } {  \sqrt{\GOR_{EU} \GRR_{UD}} + 1} \right)^2 ,
$$
and the conclusion follows.
\end{proof}

\subsection{Implied Cornfield Conditions}

The bounding factor in the last subsection implies the following Cornfield conditions:

\begin{proposition}
\label{proposition::another-corn}
We have
$$
\min(\GOR_{EU} , \GRR_{UD}) \geq \CRR_{ED} , \quad 
\max(\GOR_{EU}, \GRR_{UD})  \geq  \left(   \sqrt{  \CRR_{ED}  }  + \sqrt{  \CRR_{ED} - 1 }  \right)^2.
$$
\end{proposition}

\begin{proof}[Proof of Proposition \ref{proposition::another-corn}.]
According to Lemma \ref{lemma::increasing}, we can let $\GRR_{ED}$ goes to infinity, and obtain
$
\GOR_{EU} \geq  \CRR_{ED}.
$
Similarly, we can let $\GOR_{EU}$ goes to infinity, and obtain
$
\GRR_{UD} \geq \CRR_{ED}.
$
Combining them together, we have the following low threshold:
$
\min(\GOR_{EU} , \GRR_{UD}) \geq \CRR_{ED}.
$
According to Lemma \ref{lemma::increasing} again, we can replace $\GOR_{EU} $ and $ \GRR_{UD}$ by $\max(\GOR_{EU} , \GRR_{UD})$ in the bounding factor in Section \ref{sec::another} , and preserve the inequality as follows:
$$
\left(    \frac{   \max (\GOR_{EU},  \GRR_{UD} ) +1   }{   2 \sqrt{  \max(\GOR_{EU}, \GRR_{UD}) }   }  \right)^2 \geq  \CRR_{ED}.
$$
Solving the above inequality, we obtain
$
 \sqrt{  \max(\GOR_{EU}, \GRR_{UD}) } \geq
 \sqrt{  \CRR_{ED}  }  + \sqrt{  \CRR_{ED} - 1 }, 
$
and the high threshold follows.
\end{proof}

Propositions \ref{proposition::another} and \ref{proposition::another-corn} generalize the results of Bross \citep{Bross::1966, Bross::1967} and Lee \citep{Lee::2011} from only being applicable under the null hypothesis of no effect (i.e., only being useful for assessing how much unmeasured confounding would suffice to completely explain away an effect estimate) to alternative hypotheses and sensitivity analysis.

\subsection{Preventive Exposure}\label{sec::preventive-exposure}

For apparently preventive exposure with $\RR_{ED}^\obs < 1$, we can derive bias formula similar to Proposition \ref{proposition::another}, and we don't even need to modify the definition of $\GOR_{EU}$.

\begin{proposition}
\label{prop::preventive-another}
For apparently preventive exposure, we have
$$
\frac{  \SRR_{ED}   }{ \RR_{ED}^\obs } \leq   \left(    \frac{   \sqrt{  \GOR_{EU} \GRR_{UD}  } + 1   }{   \sqrt{\GOR_{EU}}  + \sqrt{\GRR_{UD}} }  \right)^2.
$$
\end{proposition}

\begin{proof}
[Proof of Proposition \ref{prop::preventive-another}.]
Define $\bar{E} = 1-E$. Applying Proposition \ref{proposition::another}, we have
$$
\frac{ \RR_{\bar{E}D} } {  \SRR_{\bar{E}D}   } \leq   \left(    \frac{   \sqrt{  \GOR_{\bar{E}U} \GRR_{UD}  } + 1   }{   \sqrt{\GOR_{\bar{E}U}}  + \sqrt{\GRR_{UD}} }  \right)^2.
$$
Since $\RR_{\bar{E}D} = 1/\RR_{ED}^\obs , \SRR_{\bar{E}D} = 1 / \SRR_{ED}$, and 
$$
\GOR_{\bar{E}U} = \frac{ \max_u 1/q(u)  }{ \min_u 1/q(u)}   = \frac{1/\min_u q(u)}{1/\max_u q(u)} = \frac{\max_u q(u)}{ \min_u q(u)} = \GOR_{EU},
$$
the conclusion follows.
\end{proof}

\section{Relations with Existing Results}

\subsection{Schlesselman's Formula}

For a binary confounder $U$, Schlesselman \cite{Schlesselman::1978} first obtained that
$$
\frac{\RR_{ED}^\obs}{\SRR_{ED-}}= \frac{1+(\RR_{UD|E=1} -1)\pr(U = 1 \mid E = 1)} {1+(\RR_{UD|E=0} - 1)\pr(U = 1 \mid E = 0)}.
$$
He further assumed a common relative risk of the exposure $E$ on the outcome $D$ within both $U=0$ and $U=1$, and also a common relative risk of the confounder $U$ on the outcome $D$ within both $E=0$ and $E=1$, denoted by $\gamma$. 
Under the above no-interaction assumption, Schlesselman simplified the above identity to the following formula:
\begin{eqnarray*}
\frac{ \RR_{ED}^\obs }{ \SRR_{ED} } = \frac{1 + (\gamma - 1) \pr(U=1\mid E=1)}{ 1 + (\gamma - 1) \pr(U=1\mid E=0)}.
\end{eqnarray*}
We can write $ \pr(U=1\mid E=0) =  \pr(U=1\mid E=1) / \RR_{EU}$ and then maximize the right-hand side of the above formula over $ \pr(U=1\mid E=1)$, which gives us the following inequality:
$$
\frac{ \RR_{ED}^\obs }{ \SRR_{ED} }  \leq \frac{ \RR_{EU}\times \gamma  }{\RR_{EU} + \gamma - 1}.
$$
The inequality above is the same as our main result in the main text, but is derived under unnecessary assumptions. Our result is much more general than the previous result obtained by Schlesselman \cite{Schlesselman::1978}, and his assumptions are not necessary for deriving our new bounding factor.

\subsection{Flanders and Khoury's results}

Flanders and Khoury \citep{Flanders::1990} used slightly different notation for categorical confounder $U$:
\begin{eqnarray*}
p_k &=& \pr(U=k\mid E=0),\\
\OR_k&=& \frac{  \pr(U=k\mid E=1) / \pr(U=0\mid E=1)     }{ \pr(U=k\mid E=0) / \pr(U=0\mid E=0)    } ,\\
\RR_k&=& \frac{ \pr(D=1\mid U=k, E=0) }{ \pr(D=1\mid U=0, E=0) }.
\end{eqnarray*}
They expressed the confounding relative risk for the exposed population as
$$
\CRR_{ED+} = 
\frac{  \sum_k   \RR_k \OR_k p_k   }
{  \left(   \sum_k    \OR_k p_k  \right)  \left(  \sum_k   \RR_k  p_k   \right)   }.
$$
The above sensitivity analysis formula depends on a large number of sensitivity parameters, and requires specifying the prevalence of the unmeasured confounder among unexposed population. Flanders and Khoury simplified it for binary confounder. However, for general categorical confounder, they derived the following bounds on the confounding relative risk:
$$
\CRR_{ED+}  \leq 
\min
\left\{
\frac{\max_k \OR_k }{ \sum_k \OR_k p_k },
\frac{\max_k \RR_k }{ \sum_k \RR_k p_k },
\max_k \OR_k, \max_k \RR_k,
{ 1\over p_{k^*}} , { 1\over p_{k^{**}} }
\right\},
$$
where $k^*$ and $k^{**}$ are the strata corresponding to the largest $\OR_k$ and $\RR_k$, respectively.
The upper bound depends on the prevalence of $U$. If we do not have any knowledge about the number of categories or the prevalence of $U$, the above bound reduces to
$$
\CRR_{ED+}  \leq 
\min
\left\{
\max_k \OR_k, \max_k \RR_k
\right\},
$$
which is essentially the low threshold Cornfield condition.

\section{Results for the Risk Difference Using Sensitivity Parameters on the Relative Risk Scale}
\label{sec::rd-rr}

\subsection{Lower Bounds for the Causal Risk Differences}

Define $p_1 = \pr(D=1\mid E=1)$ and $ p_0 = \pr(D=1\mid E=0)$ as the probabilities of the outcome with and without exposure, and $f = \pr(E=1)$ as the prevalence of the exposure. The causal risk differences for the exposed, unexposed and the whole population are defined as
\begin{eqnarray*}
\SRD_{ED+} &=& \pr\{  D(1)  =1 \mid E=1 \} - \pr\{  D(0) = 1 \mid E=1\} 
= p_1 - \pr\{  D(0) = 1 \mid E=1\} ,\\
\SRD_{ED-} &=& \pr\{  D(1)  =1 \mid E=0 \} - \pr\{  D(0) = 1 \mid E=0 \} 
=\pr\{  D(1)  =1 \mid E=0 \} - p_0,\\
\SRD_{ED} &=&  \pr\{  D(1)  =1\} - \pr\{  D(0) = 1\}.
\end{eqnarray*}
If $U$ suffices to control the confounding between the exposure and the outcome, then the following standardized risk differences are the causal risk differences for the exposed, unexposed and the whole population:
\begin{eqnarray*}
\SRD_{ED+} =
 p_1 -  \int r(u) F_1(du),\quad 
\SRD_{ED-} =
\int r^*(u) F_0(du) - p_0,\quad 
\SRD_{ED} =  
\int \{r^*(u) - r(u)\}F(du).
\end{eqnarray*}

\begin{proposition}
\label{prop::RD-biasfactor}
The lower bounds for the causal risk differences are
\begin{eqnarray*}
&&\SRD_{ED+}  \geq   p_1 - p_0\times \BF_U,\\
&&\SRD_{ED-} \geq  p_1/\BF_U - p_0, \\
&&\SRD_{ED} \geq   (     p_1    - p_0\times \BF_U )  \times  \left\{    f  + (1-f)/\BF_U  \right\}
=(    p_1 / \BF_U   - p_0)  \times  \left\{    f\times \BF_U + (1-f)  \right\} .
\end{eqnarray*}
\end{proposition}

\begin{proof}[Proof of Proposition \ref{prop::RD-biasfactor}.]
From the data, we can identify:
\begin{eqnarray*}
p_1 &=&  \int \pr(D=1\mid E=1, U=u) F_1(du) = \int r^*(u) F_1(du)  ,\\
p_0 &=&  \int \pr(D=1\mid E=0, U=u) F_0(du) = \int r(u) F_0(du). 
\end{eqnarray*}
However, the following two counterfactual probabilities are not identifiable:
\begin{eqnarray*}
\pr\{  D(1)=1\mid E=0 \} &=& \int \pr(D=1\mid E=1, U=u) F_0(du) = \int r^*(u) F_0(du), \\
\pr\{  D(0)=1\mid E=1 \} &=& \int \pr(D=1\mid E=0, U=u) F_1(du) = \int r(u) F_1(du) .
\end{eqnarray*} 

First, we have
\begin{eqnarray*}
\frac{ p_1  }{    \pr\{  D(1)=1\mid E=0 \}      }  
= \frac{    \int r^*(u) F_1(du)    }{     \int r^*(u) F_0(du)      }  
= \CRR_{ED-} \leq \BF_U
\end{eqnarray*}
according to Proposition \ref{prop::bound-CRR},
and thus $   \pr\{  D(1)=1\mid E=0 \}    \geq   p_1 / \BF_U$. 
Second, we have
\begin{eqnarray*}
\frac{ \pr\{  D(0)=1\mid E=1 \}      }{  p_0   } 
= \frac{  \int r(u) F_1(du)       }{  \int r(u) F_0(du) } 
= \CRR_{ED+} \leq \BF_U 
\end{eqnarray*}
according to Proposition \ref{prop::bound-CRR},
and thus $ \pr\{  D(0)=1\mid E=1 \}   \leq  p_0 \times \BF_U$.
Therefore, the lower bound for $\SRD_{ED+}$ is
$
\SRD_{ED+} \geq p_1 - p_0\times \BF_U,
$
and for $\SRD_{ED-}$ is
$
\SRD_{ED-} \geq p_1/\BF_U - p_0. 
$
We can obtain the lower bound for $\SRD_{ED}$ using $\SRD_{ED} = f \SRD_{ED+} + (1-f) \SRD_{ED-}$.
\end{proof}

If the probability of $E=1$, $f$, is unknown, the above result about $\SRD_{ED}$ is not directly useful. In the following, we obtain a lower bound for $\SRD_{ED}$ based on $\SRD_{ED} = f \SRD_{ED+} + (1-f) \SRD_{ED-}$, which does not depend on $f.$

\begin{proposition}
\label{prop::unknown-f}
We have 
$
\SRD_{ED} \geq \min(   
 p_1 - p_0\times \BF_U, 
   p_1/\BF_U - p_0
).
$
When  $p_1 > p_0$ and $1\leq \BF_U \leq \RR_{ED}^\obs$, the above lower bound reduces to
$
\SRD_{ED} \geq  p_1 - p_0\times \BF_U.
$
\end{proposition}

The above results are particularly useful for an apparently causative exposure with $\RD_{ED}^\obs > 0$, which give (possibly positive) lower bounds for the causal risk differences. However, for an apparently preventive exposure with $\RD_{ED}^\obs < 0$, we need to modify the definition of $\GRR_{EU}$ as $\GRR_{EU} = \max_u \GRR^{-1}_{EU}(u)$. And we have the following analogous results.

\begin{proposition}
\label{prop::rd-preventive}
For apparently preventive exposure with $\RD_{ED}^\obs < 0$, we have
\begin{eqnarray*}
&&\SRD_{ E D+}  \leq   p_1\times \BF_U - p_0 ,\\
&&\SRD_{ E D-} \leq  p_1  - p_0/\BF_U, \\
&&\SRD_{ E D} \leq  (   p_1\times \BF_U - p_0)  \times  \left\{    f  + (1-f)/\BF_U  \right\}
=(    p_1  - p_0/\BF_U  ) \times  \left\{    f\times \BF_U + (1-f)  \right\} .
\end{eqnarray*}
When $f$ is unknown and $1\leq \BF_U\leq 1/\RR_{ED}^\obs$, we have
$
\SRR_{ED} \leq  p_1  - p_0/\BF_U.
$
\end{proposition}

\begin{proof}
[Proof of Proposition \ref{prop::rd-preventive}.]
Define $\bar{E} = 1-E$. Applying Proposition \ref{prop::RD-biasfactor}, we have
\begin{eqnarray*}
\SRD_{\bar{E} D+} & \geq &  \pr(D=1\mid \bar{E}=1) - \pr(D=1\mid \bar{E}=0)\times \BF_U,\\
\SRD_{\bar{E}D-} &\geq & \pr(D=1\mid \bar{E}=1)/\BF_U - \pr(D=1\mid \bar{E}=0), \\
\SRD_{\bar{E}D} &\geq &  \left\{     \pr(D=1\mid \bar{E}=1)    - \pr(D=1\mid \bar{E}=0)\times \BF_U \right\}  \times  \left\{    f  + (1-f)/\BF_U  \right\}\\
&=&\left\{     \pr(D=1\mid \bar{E}=1) / \BF_U   - \pr(D=1\mid \bar{E}=0) \right\}  \times  \left\{    f\times \BF_U + (1-f)  \right\} .
\end{eqnarray*}

Since $\SRD_{\bar{E} D+}  =  - \SRD_{ ED+} , \SRD_{\bar{E}D-} = -  \SRD_{ED-}$ and $ \SRD_{\bar{E}D}  = - \SRD_{ED} $, the first three conclusions follow.
When $f$ is unknown and $1\leq \BF_U\leq 1/\RR_{ED}^\obs$, we have
$
\SRD_{ED} \leq \max(  \SRD_{ED+} , \SRD_{ED-} )
 = p_1  - p_0/\BF_U.
$

\end{proof}

The above discussion is within strata of observed covariates $C$.  All probabilities are essentially conditional probabilities, e.g., $\pr(D=1\mid E=1, C=c), \pr(E=1\mid C=c), etc.$ Consequently, the bounding factor and causal risk differences are also conditional, denoted as $\BF_{U|c}, \SRD_{ED|c+}, \SRD_{ED|c-}$ and $ \SRD_{ED|c}$. Due to the linearity of the risk difference, i.e., $\SRD_{ED+} = \sum_c \SRD_{ED|c+} \pr(C=c\mid E=1), \SRD_{ED-} = \sum_c \SRD_{ED|c-} \pr(C=c\mid E=0)$ and $\SRD_{ED} = \sum_c \SRD_{ED|c} \pr(C=c)$, we have the following results about the marginal risk differences:
\begin{eqnarray*}
\SRD_{ED+} &\geq &  \sum_c \left\{  \pr(D=1\mid E=1, C=c) - \pr(D=1\mid E=0, C=c)\times \BF_{U|c} \right\} \pr(C=c\mid E=1) , \\
\SRD_{ED-} &\geq &  \sum_c \left\{  \pr(D=1\mid E=1, C=c) / \BF_{U|c} - \pr(D=1\mid E=0, C=c)  \right\}  \pr(C=c\mid E=0),\\
\SRD_{ED} &\geq & f \sum_c \left\{  \pr(D=1\mid E=1, C=c) - \pr(D=1\mid E=0, C=c)\times \BF_{U|c} \right\} \pr(C=c\mid E=1) \\
&&+(1-f)   \sum_c \left\{  \pr(D=1\mid E=1, C=c) / \BF_{U|c} - \pr(D=1\mid E=0, C=c)  \right\}  \pr(C=c\mid E=0).
\end{eqnarray*}

\subsection{Statistical Inference for the Causal Risk Differences}

In previous subsections we discussed the population quantities assuming that we knew the distribution of $(E,D,C)$. In this subsection, we will discuss the finite sample inference for the causal risk differences. 
We can straightforwardly estimate $f, p_1$ and $p_0$ by sample frequencies $\widehat{f}, \widehat{p}_1$ and $\widehat{p}_0$ with standard errors $s, s_1$ and $s_0$, respectively. Then we can estimate the lower bound for $\SRD_{ED+}$ by $\widehat{p}_1 - \widehat{p}_0\times \BF_U$ with standard error
$  
(
s_1^2 + s_0^2 \times \BF_U^2
)^{1/2},
$
 estimate the lower bound for $\SRD_{ED-}$ by $\widehat{p}_1/\BF_U - \widehat{p}_0$ with standard error
$
(
s_1^2 / \BF_U^2 + s_0^2
)^{1/2},
$
and estimate the lower bound for $\SRD_{ED}$ by $(\widehat{p}_1 - \widehat{p}_0\times \BF_U)\times \{ \widehat{f} + (1-\widehat{f}) / \BF_U \} $ or $  (\widehat{p}_1/\BF_U - \widehat{p}_0)\times \{  \widehat{f}\times \BF_U + (1-\widehat{f})  \} $ with standard error
$$ 
\sqrt{  
(s_1^2 + s_0^2 \times \BF_U^2 )   \left(  \widehat{f} + \frac{ 1-\widehat{f} } {  \BF_U  } \right) ^2
+ (\widehat{p}_1 - \widehat{p}_0\times \BF_U)^2     (1-\BF_U^{-1})^2 s^2
},
$$
using a standard argument of the delta-method. After obtaining the point estimates and their standard errors, we can construct confidence intervals for these causal risk differences.

Note that even without estimating the prevalence, $f$, of the exposure, if the exposure is apparently causative, we can use the lower bound of $\min(\SRD_{ED+},\SRD_{ED-})$ as a lower bound for $\SRD_{ED}$. The point estimate of the causal risk difference averaged over the observed covariates can be obtained by the weighted average of the point estimates of the causal risk differences within strata of $C$ with the proportions of the strata as the weights, and the corresponding sampling variance is the weighted average of the sampling variances within strata with the squared proportions of the strata as the weights.

\subsection{Implied Cornfield Conditions}

The results in Proposition \ref{prop::RD-biasfactor} imply the following Cornfield conditions.

\begin{proposition}
\label{prop::cornfield-rd}
For an unmeasured confounder to reduce the observed risk difference to be $\SRD_{ED+}, \SRD_{ED-}$ and $\SRD_{ED}$ respectively, the joint Cornfield conditions are
\begin{eqnarray*}
&&\BF_U \geq   (p_1 - \SRD_{ED+} ) / p_0 ,\\
&&\BF_U \geq   p_1 / (p_0+ \SRD_{ED-}) ,\\
&&\BF_U \geq     \frac{   \sqrt{   \{  \SRD_{ED} + p_0(1-f) -  p_1f   \}^2+4p_1p_0f(1-f)       }  -    \{  \SRD_{ED} + p_0(1-f) - p_1f   \} }{   2p_0f }.
\end{eqnarray*}
\end{proposition}

\begin{proof}
[Proof of Proposition \ref{prop::cornfield-rd}.]
It is straightforward to see that the first two conclusions of Proposition \ref{prop::RD-biasfactor} imply the first two inequalities. From the third conclusion of Proposition \ref{prop::RD-biasfactor}, we have the following quadratic inequality about $\BF_U$:
$$
(p_0f) \BF_U^2 +  \{    p_0(1-f) + \SRD_{ED} - p_1f    \} \BF_U - p_1(1-f) \geq 0. 
$$
The corresponding equation has one negative root and the following positive root:
$$
\BF_U^* = \frac{   \sqrt{   \{  \SRD_{ED} + p_0(1-f) -  p_1f   \}^2+4p_1p_0f(1-f)       }  -    \{  \SRD_{ED} + p_0(1-f) - p_1f   \} }{   2p_0f }.
$$
Since $\BF_U>0$, the inequality has the solution $\BF_U \geq \BF_U^*.$
\end{proof}

Similar to the discussion in the last two sections, we can also derive the low and high threshold Cornfield conditions from the above joint Cornfield conditions for $(\GRR_{EU}, \GRR_{UD})$.
If $\SRD_{ED+}, \SRD_{ED-}$ and $\SRD_{ED}$ are zero, all the conditions in Proposition \ref{prop::cornfield-rd} reduce to $\BF_U \geq \RR_{ED}^\obs$, the one derived from the result about the relative risk of the exposure on the outcome. Therefore, the formula from the risk difference is the same as that derived from the relative risk under the null hypothesis, but they are different under the alternative hypotheses.

With finite sample, we can also find the smallest bounding factor that can reduce the lower confidence limit of the lower bound of the causal risk differences to a certain magnitude. 
We will discuss $(1-\alpha)\%$ confidence intervals based on asymptotic normality, and let $z_{\alpha} = \Phi^{-1}(1-\alpha/2)$ denote the upper $\alpha/2$ quantile of the standard normal distribution (e.g., when $\alpha=0.05$, $z_{0.05}=1.96$). In order to reduce the confidence interval of the risk difference on the exposed to cover a true causal risk difference $\SRD_{ED+}$, the bounding factor must satisfy
$$
\widehat{p}_1 - \widehat{p}_0\times \BF_U - z_\alpha 
\sqrt{ 
s_1^2 + s_0^2 \times \BF_U^2
} \leq \SRD_{ED+},
$$
which has the following solution:
\begin{eqnarray}
&\BF_U \geq 
 \frac{   \widehat{p}_0 (\widehat{p}_1 - \SRD_{ED+})  - \sqrt{     \widehat{p}_0^2 (\widehat{p}_1 - \SRD_{ED+})^2 -   (\widehat{p}_0^2 - z_{\alpha}^2 s_0^2 )  \{    ( \widehat{p}_1  + \SRD_{ED+} )^2 -  z_{\alpha}^2 s_1^2 \}      }       }
{ \widehat{p}_0^2 - z_{\alpha}^2 s_0^2  } .&\label{eq::solve1}
\end{eqnarray}
In order to reduce the confidence interval of the risk difference on the unexposed to cover a true causal risk difference $\SRD_{ED-}$, the bounding factor must satisfy
$$
\widehat{p}_1/\BF_U - \widehat{p}_0 - z_\alpha 
\sqrt{ 
s_1^2 / \BF_U^2 + s_0^2
} \leq \SRD_{ED-},
$$
which has the following solution:
\begin{eqnarray}
&\BF_U \geq
\frac{          \widehat{p}_1(\widehat{p}_0 + \SRD_{ED-})    -       \sqrt{     \widehat{p}_1^2(\widehat{p}_0 + \SRD_{ED-})^2 -  \{  (\widehat{p}_0 + \SRD_{ED-})^2 - z_\alpha^2 s_0^2  \}   (\widehat{p}_1^2 - z_\alpha^2 s_1^2)       }     }
{    (\widehat{p}_0 + \SRD_{ED-})^2 - z_\alpha^2 s_0^2        } . & \label{eq::solve2}
\end{eqnarray}
Note that if we assume $\SRD_{ED+} =\SRD_{ED-} =0$, the above solutions in (\ref{eq::solve1}) and (\ref{eq::solve2}) reduce to the same form:
$$
\BF_U \geq
\frac{          \widehat{p}_1\widehat{p}_0     -       \sqrt{     \widehat{p}_1^2\widehat{p}_0^2 -  (   \widehat{p}_0^2 - z_\alpha^2 s_0^2 ) (\widehat{p}_1^2 - z_\alpha^2 s_1^2)       }     }
{    \widehat{p}_0 ^2 - z_\alpha^2 s_0^2        } .
$$
In order to reduce the confidence interval of the risk difference to cover a true causal risk difference $\SRD_{ED}$, the bounding factor must satisfy
\begin{eqnarray} 
&&(\widehat{p}_1 - \widehat{p}_0\times \BF_U) \left(   \widehat{f} + \frac{ 1-\widehat{f} } {  \BF_U}  \right)  \nonumber  \\
&& - z_\alpha 
\sqrt{  
(s_1^2 + s_0^2 \times \BF_U^2 )   \left(  \widehat{f} + \frac{ 1-\widehat{f} } {  \BF_U  } \right) ^2
+ (\widehat{p}_1 - \widehat{p}_0\times \BF_U)^2     (1-\BF_U^{-1})^2 s^2
}  \leq  \SRD_{ED}, \nonumber\\ \label{eq::solve3}
\end{eqnarray} 
which can be solved numerically. For example, we can apply a grid search for the solution of (\ref{eq::solve3}) over the following bounded range:
$$
\BF_U\in \left(  1,
 \frac{   \sqrt{   \{  \SRD_{ED} + \widehat{p}_0(1-\widehat{f}) -  \widehat{p}_1 \widehat{f}   \}^2+4\widehat{p}_1 \widehat{p}_0f(1-\widehat{f})       }  -    \{  \SRD_{ED} + \widehat{p}_0(1-\widehat{f}) - \widehat{p}_1 \widehat{f}   \} }{   2\widehat{p}_0 \widehat{f} }
\right),
$$
since the point estimate has already been reduced to $\SRD_{ED}$ when $\BF_U$ attains the above upper bound of range.

\section{SAS Code for the Risk Difference Using Sensitivity Parameters on the Relative Risk Scale}

In this section, we provide SAS code for sensitivity analysis on the risk difference scale. The SAS code here illustrates analysis using logistic regression for a binary outcome as this is an approach that is commonly employed.

%

Suppose we have a dataset named ``leadlogit'' with variables lead, smoking, age, male. Suppose we are interested in the risk difference of smoking on the high blood lead level at the covariate level, age = 50 and male = 1.

%
%
%
%
%
%
%
%
%

To implement sensitivity analysis for risk difference we need to obtain point estimate and standard error for $f = \pr(E=1)$, which can be done via the following SAS code.


{\setlength{\baselineskip}{0.5\baselineskip}
\scriptsize
\begin{verbatim}
proc means data=lib.leadlogit ; /*f and se(f)*/
	var smoking;
    output out=sumstat mean=mean var=var N=N;
run;

data sumstat (KEEP=MEAN SE);
	set sumstat;
	se=(var/N)**0.5;
run;
\end{verbatim}
}

The following code obtains the predicted probabilities $p_{1|c} = \pr\{ Y(1)=1\mid C=c\}$ and $p_{0|c}=\pr\{ Y(0)=1\mid C=c\}$ with standard errors. 


{\setlength{\baselineskip}{0.5\baselineskip}
\scriptsize
\begin{verbatim}
proc logistic data = lib.leadlogit;/*predict probs*/
	model lead = smoking age male;
	score data = lib.leadlogit_new out=logit_pred clm;
run;

proc contents data =logit_pred ;
run;

data logit_pred (keep=P_TRUE se_p);/*p1 p0 se(p1) se(p0)*/
	set logit_pred;

	logit_LCL_TRUE	=log(LCL_TRUE/(1-LCL_TRUE));
	logit_P_TRUE	=log(P_TRUE/(1-P_TRUE));
	logit_UCL_TRUE 	=log(UCL_TRUE /(1-UCL_TRUE ));

	se_eta  =(logit_UCL_TRUE-logit_LCL_TRUE)/2/1.96;
	se_p	=P_TRUE**2/EXP(logit_P_TRUE)*se_eta;
run;
\end{verbatim}
}

%

In the following SAS code, we need to input from line 2 to line 7 the point estimates and standard errors of the prevalence $f $, and the two predicted outcome probabilities $p_{1|c}  $ and $p_{0|c}  $. The output contains lower bounds for the point estimates and confidence intervals of the causal risk differences for the exposed, unexposed and the whole population. Figure \ref{fg::SAS-output} is the SAS output for the causal risk difference estimates for the whole population. 
For other problems, we need to change the numbers from line 2 to line 7 accordingly. We can also change the measures of the strength of confounding in lines 8 and 9. The output from SAS will be similar to the one shown in Figure \ref{fg::SAS-output}.

{\setlength{\baselineskip}{0.5\baselineskip}
\scriptsize
\begin{verbatim}
proc iml;/*Sensitivity analysis without assumptions for RD*/
f             = 0.2032934132;/*point and interval estimate of prevalence and response rates*/
p1            = 0.101645862;
p0            = 0.0398930775;
s2_f          = 0.0069647038;
s2_p1         = 0.0147497019;
s2_p0         = 0.0058931321;
RR_EU = {1.2 1.3 1.5 1.8 2 2.5 3 5};/*strenghth of confounding*/
RR_UD = {1.2 1.3 1.5 1.8 2 2.5 3 5};
rownames_EU = CHAR(RR_EU, NCOL(RR_EU), 1);
colnames_UD = CHAR(RR_UD, NCOL(RR_UD), 1);
BiasFactor  = J(NCOL(RR_EU), NCOL(RR_UD), 1);
SPACE       = J(NCOL(RR_EU), NCOL(RR_UD), " ");
LeftP       = J(NCOL(RR_EU), NCOL(RR_UD), "(");
Mid         = J(NCOL(RR_EU), NCOL(RR_UD), ",");
RightP      = J(NCOL(RR_EU), NCOL(RR_UD), ")");
/*initial values*/
RD_exposed   = BiasFactor;
RD_exposed_L = BiasFactor;
RD_exposed_U = BiasFactor;
RD_unexposed   = BiasFactor;
RD_unexposed_L = BiasFactor;
RD_unexposed_U = BiasFactor;
RD_whole   = BiasFactor;
RD_whole_L = BiasFactor;
RD_whole_U = BiasFactor;
W_whole = BiasFactor;
Var_exposed = BiasFactor;
Var_unexposed = BiasFactor;
Var_whole = BiasFactor;
/*Sensitivity analysis*/
DO i=1 TO NCOL(RR_EU);
      Do j=1 to NCOL(RR_UD);
	      BiasFactor[i, j] = RR_EU[i]*RR_UD[j]/(RR_EU[i] + RR_UD[j] - 1);
		  /*exposed*/
          RD_exposed[i, j]     = p1 - p0*BiasFactor[i, j];
          Var_exposed[i, j]    = s2_p1 + s2_p0*(BiasFactor[i, j])**2;
          RD_exposed_L[i, j]   = RD_exposed[i, j] - 1.96*sqrt(Var_exposed[i, j]);
		  RD_exposed_U[i, j]   = RD_exposed[i, j] + 1.96*sqrt(Var_exposed[i, j]);
          /*exposed*/
          RD_unexposed[i, j]   = p1/BiasFactor[i, j] - p0;
		  Var_unexposed[i, j]  = s2_p1/(BiasFactor[i, j])**2 + s2_p0;
          RD_unexposed_L[i, j] = RD_unexposed[i, j] - 1.96*sqrt(Var_unexposed[i, j]);
		  RD_unexposed_U[i, j] = RD_unexposed[i, j] + 1.96*sqrt(Var_unexposed[i, j]);
          /*whole*/
		  W_whole[i, j]        = f + (1-f)/BiasFactor[i, j];
          RD_whole[i, j]       = RD_exposed[i, j]*W_whole[i, j];
          Var_whole[i, j]      = Var_exposed[i, j]*(W_whole[i, j])**2 
                                           + (RD_exposed[i, j])**2*(1-1/BiasFactor[i, j])**2*s2_f;
          RD_whole_L[i, j]     = RD_whole[i, j] - 1.96*sqrt(Var_whole[i, j]);
          RD_whole_U[i, j]     = RD_whole[i, j] + 1.96*sqrt(Var_whole[i, j]);
       END;
END;
/*print;*/
RD_exposed = CATX(" ", CHAR(round(RD_exposed, 0.0001)), LeftP, CHAR(round(RD_exposed_L, 0.0001)), Mid, 
                          CHAR(round(RD_exposed_U, 0.0001)), RightP);
print RD_exposed[colname = colnames_UD
                 rowname = rownames_EU
                 label = "Bounds on corrected estimates and confidence intervals for risk difference 
                 among exposed (columns correspond to increasing strength of the risk ratio of U 
                 on the outcome; rows correspond to increasing strength of risk ratio 
                 relating the exposure and U)"];
RD_unexposed = CATX(" ", CHAR(round(RD_unexposed, 0.0001)), LeftP, CHAR(round(RD_unexposed_L, 0.0001)), Mid, 
                               CHAR(round(RD_unexposed_U, 0.0001)), RightP);
print RD_unexposed[colname = colnames_UD
                   rowname = rownames_EU
                   label = "Bounds on corrected estimates and confidence intervals for risk difference 
                   among unexposed (columns correspond to increasing strength of the risk ratio of U 
                   on the outcome; rows correspond to increasing strength of risk ratio 
                   relating the exposure and U)"];
RD_whole = CATX(" ", CHAR(round(RD_whole, 0.0001)), LeftP, CHAR(round(RD_whole_L, 0.0001)), Mid, 
                              CHAR(round(RD_whole_U, 0.0001)), RightP);
print RD_whole[colname = colnames_UD
               rowname = rownames_EU
               label = "Bounds on corrected estimates and confidence intervals for risk difference 
               among the whole population (columns correspond to increasing strength of the risk ratio of U 
               on the outcome; rows correspond to increasing strength of risk ratio 
               relating the exposure and U)"];
run;
\end{verbatim}
}

\begin{figure}
\includegraphics[width=\textwidth]{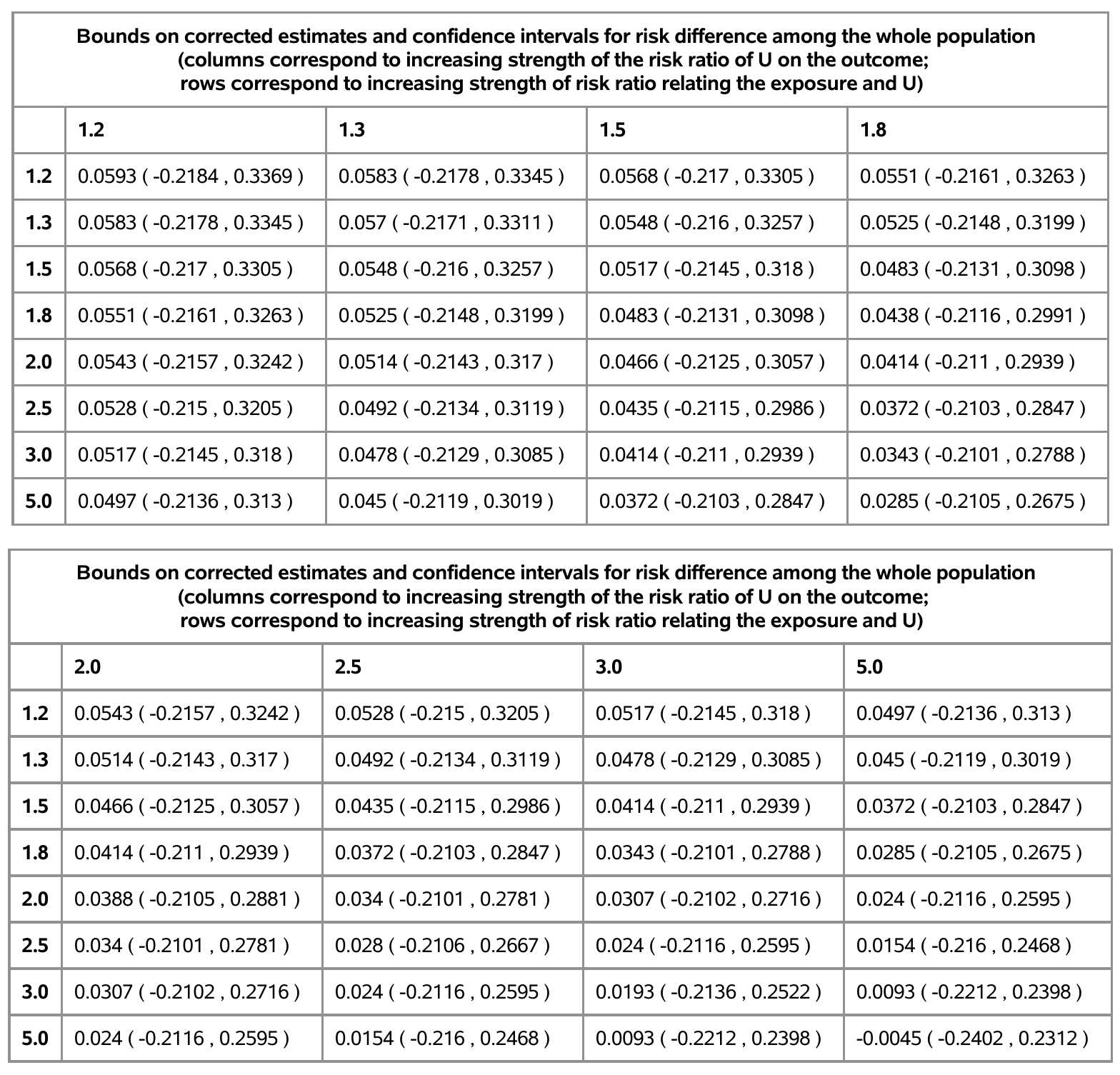}
\caption{SAS Output of Sensitivity Analysis on the Risk Difference Scale for the Whole Population}\label{fg::SAS-output}
\end{figure}

\section{Results for the Risk Difference Using Sensitivity Parameters on the Risk Difference Scale}

\subsection{A Useful Proposition}

We first recall some definitions in the main text, and assume a categorical unmeasured confounder $U$. Let $\RD_{ED}^\obs = \pr(D=1\mid E=1) -  \pr(D=1\mid E=0)$ denote the observed risk difference,
\begin{eqnarray*}
\SRD_{ED} = \sum_{k=0}^{K-1} \{ \pr(D=1\mid E=1, U=k) - \pr(D=1\mid E=0, U=k)\} \pr(U=k)
\end{eqnarray*}
denote the true causal risk difference, and $\CRD_{ED} =  \RD_{ED}^\obs -  \SRD_{ED}   $ denote the confounding risk difference of the exposure $E$ on the outcome $D$.
Define $\alpha_k = \pr(U=k\mid E=1) - \pr(U=k\mid E=0) $ and $\GRD_{EU} = \max_{k\geq 1}  |\alpha_k|$. Define $\beta_{k}^* =  \pr(D=1\mid E=1, U=k) - \pr(D=1\mid E=1, U=0) $ and $\beta_{k}   =  \pr(D=1\mid E=0, U=k) - \pr(D=1\mid E=0, U=0)  $. Define $\GRD_{UD|E=1}=\max_{k\geq 1}  |\beta_{k}^*|, \GRD_{UD|E=0} = \max_{k\geq 1}  |\beta_{k}|$ and $\GRD_{UD} = \max(\GRD_{UD|E=1}, \GRD_{UD|E=0})$.
The confounding risk difference can be decomposed as follows. 

\begin{proposition}\label{lemma::CRD}
The confounding risk difference of $E$ on $D$, $\CRD_{ED}$, can be expressed as
$$
\CRD_{ED} = \RD_{ED}^\obs - \SRD_{ED} =   \sum_{k=1}^{K-1} \alpha_k  \{  \beta_k^* \pr(E=0) + \beta_k \pr(E=1) \}.
$$
\end{proposition}

\noindent {\it Proof of Proposition \ref{lemma::CRD}.}
The true and observed risk differences of $E$ on $D$ can be expressed as
\begin{eqnarray*}
\SRD_{ED} &=& \sum_{k=0}^{K-1}     \pr(D=1\mid E=1, U=k) \pr(U=k)    - \sum_{k=0}^{K-1}   \pr(D=1\mid E=0, U=k)  \pr(U=k), \\
\RD_{ED}^\obs &=&  \sum_{k=0}^{K-1} \pr( D=1\mid E=1, U=k ) \pr(U=k\mid E=1)  - \sum_{k=0}^{K-1} \pr(D=1\mid E=0, U=k) \pr(U=k\mid E=0).
\end{eqnarray*}
Therefore, the confounding risk difference of $E$ on $D$, $\CRD_{ED}$, can be expressed as
\begin{eqnarray*}
\CRD_{ED}  
&=& \sum_{k=0}^{K-1} \pr( D=1\mid E=1, U=k ) \{   \pr(U=k\mid E=1) - \pr(U=k) \}  \\
&& -\sum_{k=0}^{K-1} \pr(D=1\mid E=0, U=k) \{  \pr(U=k\mid E=0)  - \pr(U=k) \}  .
\end{eqnarray*}

Applying the law of total probability, we have the following results:
$$
\pr(U=k\mid E=1) - \pr(U=k) = \alpha_k  \pr(E=0), \quad 
\pr(U=k\mid E=0) - \pr(U=k)= - \alpha_k  \pr(E=1).
$$
Therefore, the confounding risk difference can be rewritten as
\begin{eqnarray*}
\CRD_{ED} &=& \sum_{k=0}^{K-1} \alpha_k \pr(D=1\mid E=1, U=k) \pr(E=0) +  \sum_{k=0}^{K-1} \alpha_k \pr(D=1\mid E=0, U=k) \pr(E=1)\\
&=&\sum_{k=0}^{K-1} \alpha_k  \{    \pr(D=1\mid E=1, U=k) \pr(E=0) +    \pr(D=1\mid E=0, U=k) \pr(E=1)   \}.
\end{eqnarray*}
Using the fact that $\alpha_0 = -\sum_{k=1}^{K-1} \alpha_k$, we can obtain that
\begin{eqnarray*}
\CRD_{ED} &=&\sum_{k=1}^{K-1} \alpha_k  \{    \pr(D=1\mid E=1, U=k) \pr(E=0) +    \pr(D=1\mid E=0, U=k) \pr(E=1)   \} \\
&&- \sum_{k=1}^{K-1} \alpha_k  \{    \pr(D=1\mid E=1, U=0) \pr(E=0) +    \pr(D=1\mid E=0, U=0) \pr(E=1)   \} \\
&=& \sum_{k=1}^{K-1} \alpha_k  \{  \beta_k^* \pr(E=0) + \beta_k \pr(E=1) \}. \Box
\end{eqnarray*}

\subsection{Binary Confounder}

For a binary confounder $U$ with $K=2$, we have the following proposition.

\begin{proposition}
\label{prop::binary-RD}
When $U$ is binary, we have $\GRD_{EU}\times \GRD_{UD}\geq   \RD_{ED}^\obs - \SRD_{ED} $, implying
$$
\min\left(    \RD_{EU} ,    \GRD_{UD}    \right)  \geq   \RD_{ED}^\obs - \SRD_{ED} ,\quad 
\max \left(   \RD_{EU} ,   \GRD_{UD}    \right)  \geq  \sqrt{\RD_{ED}^\obs - \SRD_{ED}}.
$$
\end{proposition}

\begin{proof}[Proof of Proposition \ref{prop::binary-RD}.]
We have
\begin{eqnarray*}
\CRD_{ED} &=& \alpha_1 \{ \beta_{11} \pr(E=0) + \beta_{01} \pr(E=1) \} 
= \RD_{EU} \{   \RD_{UD|E=1} \pr(E=0) + \RD_{UD|E=0}\pr(E=1)    \} .
\end{eqnarray*}
Since $\CRD_{ED} \geq 0$ and $\RD_{EU} \geq 0$, we have 
$
 \RD_{UD|E=1} \pr(E=0) + \RD_{UD|E=0}\pr(E=1)  \geq 0.
 $ 
Therefore, $ \RD_{UD|E=1} $ and $ \RD_{UD|E=0}$ cannot both be negative, and thus we have 
$$
\RD_{UD|E=1} \pr(E=0) + \RD_{UD|E=0}\pr(E=1) <   \max(  \RD_{UD|E=1}  ,  \RD_{UD|E=0} ) = \GRD_{UD}  .
$$ 
Therefore,
$
 \CRD_{ED}  \leq  \RD_{EU}  \times   \GRD_{UD}   , 
$
which implies that
$
\min\left(    \RD_{EU} ,    \GRD_{UD}    \right)  \geq  \CRD_{ED} = \RD_{ED}^\obs - \SRD_{ED} ,$ and 
$\max \left(   \RD_{EU} ,   \GRD_{UD}    \right)  \geq  \sqrt{  \CRD_{ED}  }  = \sqrt{\RD_{ED}^\obs - \SRD_{ED}}.$
\end{proof}

\subsection{General Categorical Confounder}

For categorical confounder $U$, no simple form of the bounding factor is available, but we can still show that $\GRD_{EU}$ and $\GRD_{UD} $ must satisfy the following conditions:

\begin{proposition}\label{prop::RD-no-mono}
For a categorical confounder $U$, we have
\begin{eqnarray*}
&&\GRD_{EU} \geq  (\RD_{ED}^\obs -  \SRD_{ED} ) /(K-1), \\
&&\GRD_{UD} \geq (\RD_{ED}^\obs -  \SRD_{ED} )   /2,  \\
&& \max(\GRD_{EU},  \GRD_{UD}   )  \geq  \max\left\{  \sqrt{  (\RD_{ED}^\obs -  \SRD_{ED} )   /(K-1) } , (\RD_{ED}^\obs -  \SRD_{ED} ) /2    \right\} .
\end{eqnarray*}
\end{proposition}
When $K=3$ such as a three-level genetic confounder, these conditions reduce to
\begin{eqnarray*}
\min(\GRD_{EU} , \GRD_{UD} ) \geq  (\RD_{ED}^\obs -  \SRD_{ED} )  /2,\quad 
\max(\GRD_{EU} , \GRD_{UD} ) \geq \sqrt{ (\RD_{ED}^\obs -  \SRD_{ED} )  /2}.
\end{eqnarray*}

\begin{proof}[Proof of Proposition \ref{prop::RD-no-mono}.]
Since 
\begin{eqnarray*}
\CRD_{ED} &=& \Big|  \sum_{k=1}^{K-1} \alpha_k \{  \beta_k^* \pr(E=0) + \beta_k\pr(E=1)  \} \Big|   
\leq  \GRD_{EU}   \sum_{k=1}^{K-1}  |  \beta_k^* \pr(E=0) + \beta_k\pr(E=1)  |  \\
&\leq &   \GRD_{EU}\sum_{k=1}^{K-1}  \max( | \beta_k^*|, |  \beta_k|  )  
 \leq \GRD_{EU}(K-1),
\end{eqnarray*}
we have $\GRD_{EU}\geq \CRD_{ED}/(K-1)$. The equality is attainable if and only if (c1) $\alpha_k=\CRD_{ED}/(K-1)$, and $ \beta_k^* = \beta_k = 1$ for $k=1,\ldots,(K-1)$; or (c2) $\alpha_k=-1$, and $\beta_k^* = \beta_k = -1$ for $k=1,\ldots,K.$
The condition (c1) requires that the risk difference of the exposure $E$ on each category of $U$ to be the same as $\CRD_{ED}/(K-1)$, and the confounder $U$ is a perfect predictor of the disease $D$. Similar interpretation applies to condition (c2).

Since 
\begin{eqnarray*}
\CRD_{ED} &=&  \Big|  \sum_{k=1}^{K-1} \alpha_k \{  \beta_k^* \pr(E=0) + \beta_k\pr(E=1)  \} \Big|  
 \leq      \sum_{k=1}^{K-1}  | \alpha_k |   \max( | \beta_k^*|, |  \beta_k|  )    \leq \GRD_{UD} \sum_{k=1}^{K-1} |\alpha_k| \\
& \leq& \GRD_{UD} \sum_{k=1}^{K-1} \pr(U=k\mid E=1) + \GRD_{UD}  \sum_{k=1}^{K-1} \pr(U=k\mid E=0) 
\leq  2\GRD_{UD},
\end{eqnarray*}
the lower bound for $\GRD_{UD}$ is $\GRD_{UD}\geq \CRD_{ED}/2.$
The equality is attainable if and only if $\pr(U=0\mid E=0) = \pr(U=0\mid E=1)= 0, \pr(U=k\mid E=1)  \pr(U=k\mid E=0) = 0$ for $k=1,...,(K-1)$, and $\beta_k^* = \beta_k = \pm \CRD_{ED}/2$ with the sign the same as the sign of $\alpha_k$.

Since $\CRD_{ED} \leq (K-1) \GRD_{EU}\GRD_{UD} \leq (K-1)\max^2(\GRD_{EU},\GRD_{UD})$, we have $\max(\GRD_{EU},\GRD_{UD})\geq  \sqrt{  \CRD_{ED}/(K-1) }$, with the equality attainable if and only if $\alpha_k = \beta_k^*  = \beta_k = \pm  \sqrt{  \CRD_{ED}/(K-1)  }$ for $k=1,\ldots,K-1$. Due to the constraint $\sum_{k=1}^{K-1} | \alpha_k | \leq 2$ discussed above, the equality is attainable if and only if $(K-1)\sqrt{  \CRD_{ED}/(K-1)  } \leq 2$ or $(K-1) \CRD_{ED}\leq 4$. When $(K-1)\CRD_{ED} > 4$, $\GRD_{UD}$ can attain its lower bound $\CRD_{ED}$ with $\sum_{k=1}^{K-1} |\alpha_k| = 2.$ Therefore, $\GRD_{EU}$ can attain its lower bound $2/(K-1)$, which, in this case, is smaller than $\CRD_{ED}/2.$
In summary, the lower bound for $\max(\GRD_{EU},\GRD_{UD})$ is $\max(\GRD_{EU},\GRD_{UD}) \geq \sqrt{   \CRD_{ED} / (K-1)  },$ if $(K-1)\CRD_{ED} \leq  4$, and $\max(\GRD_{EU},\GRD_{UD}) \geq \CRD_{ED} / 2$, if $(K-1)\CRD_{ED} > 4$. Equivalently, we have 
$
\max(\GRD_{EU},\GRD_{UD}) \geq \max\left\{      \sqrt{ \CRD_{ED}/(K-1) } , \CRD_{ED}/2 \right\} .
$
\end{proof}

For the Cornfield conditions for the risk difference, sharper conditions can be obtain by imposing the monotonicity assumption that
$\alpha_k \geq 0$ for $k=1, \cdots, (K-1)$. It requires that each non-zero category of $U$ is more prevalent under exposure, which is naturally satisfied for binary confounder. 
Under the monotonicity assumption, the conditions for the risk difference can be strengthened.

\begin{proposition}\label{prop::RD-mono}
For a categorical confounder under monotonicity, we have that
\begin{eqnarray*}
&&\GRD_{EU}    \geq (\RD_{ED}^\obs -  \SRD_{ED} )/(K-1), \\
&&\GRD_{UD} \geq  \RD_{ED}^\obs -  \SRD_{ED} ,  \\
&&\max(\GRD_{EU}, \GRD_{UD} ) \geq  \max\left\{   \sqrt{  (\RD_{ED}^\obs -  \SRD_{ED} ) /(K-1) },   \RD_{ED}^\obs -  \SRD_{ED} \right\}  .
\end{eqnarray*}
\end{proposition}

\begin{proof}Proof of Proposition \ref{prop::RD-mono}.
The bound for $\GRD_{EU}$ remains the same.
Since 
$$
\CRD_{ED}  = \Big|  \sum_{k=1}^{K-1} \alpha_k \{  \beta_k^* \pr(E=0) + \beta_k
(E=1)  \} \Big|  
\leq \GRD_{UD} \sum_{k=1}^{K-1} |\alpha_k| \leq \GRD_{UD} (-\alpha_0)  
\leq  \GRD_{UD},
$$
the lower bound for $\GRD_{UD}$ is $\GRD_{UD}\geq \CRD_{ED}$
The equality is attainable if and only if $\alpha_0 = -1$ and $\beta_k^*= \beta_k = \CRD_{ED}$ for $k=1,\ldots,K-1$. The condition requires that the presence or absence of the confounder $U$ is perfectly predictive to the exposure $E$, and each category of $U$ is equally predictive to the disease $D$.

Since $\CRD_{ED} \leq (K-1) \GRD_{EU}\GRD_{UD} \leq (K-1)\max^2(\GRD_{EU},\GRD_{UD})$, we have $\max(\GRD_{EU},\GRD_{UD})\geq \sqrt{  \CRD_{ED}/(K-1)  }$, with the equality attainable if and only if $\alpha_k = \beta_k^*  = \beta_k = \pm \sqrt{  \CRD_{ED}/(K-1)  }$ for $k=1,\ldots,K-1$. Due to the constraint $\sum_{k=1}^{K-1}  \alpha_k = -\alpha_0  \leq 1$ discussed above, the equality is attainable if and only if $(K-1)\sqrt{  \CRD_{ED}/(K-1)  } \leq 1$ or $(K-1) \CRD_{ED}\leq 1$. When $(K-1)\CRD_{ED} > 1$, $\GRD_{UD}$ can attain its lower bound $\CRD_{ED}$ with $\sum_{k=1}^{K-1} \alpha_k = 1.$ Therefore, $\GRD_{EU}$ can attain its lower bound $1/(K-1)$, which, in this case, is smaller than $\CRD_{ED}.$
In summary, the lower bound for $\max(\GRD_{EU},\GRD_{UD})$ is $\max(\GRD_{EU},\GRD_{UD}) \geq \sqrt{  \CRD_{ED} / (K-1)  },$ if $(K-1)\CRD_{ED} \leq  1$, and $\max(\GRD_{EU},\GRD_{UD}) \geq \CRD_{ED} $, if $(K-1)\CRD_{ED} > 1$. Equivalently, we have 
$
\max(\GRD_{EU},\GRD_{UD}) \geq \max\left\{    \sqrt{  \CRD_{ED}/(K-1)  } , \CRD_{ED} \right\} .
$
\end{proof}

The results in Propositions \ref{lemma::CRD} to \ref{prop::RD-mono} generalize previous results \citep{Ding::2014} from the null hypothesis of no effect ($\SRD_{ED} = 0$) to alternative hypotheses ($\SRD_{ED}$ arbitrary).

\section{A Bounding Factor for Rare Time-to-Event Outcome on the Hazard Ratio Scale}

Let $f, S, \lambda$ be the probability density, survival function and hazard function of a positive continuous outcome $T$. The outcome is rare in the sense that $\pr(T\leq \mathcal{T}) $ is not much greater than $ 0$, where $\mathcal{T}$ is the time point of the end our research of interest. In the following, we will always make the rare outcome assumption. Although $f, S, \lambda$ are defined on the whole positive real line, our interest only within interval $[0,\mathcal{T}]$. Let $U$ be another random variable, and $f(t\mid u), S(t\mid u), \lambda(t\mid u)$ are the conditional probability density, survival function, and hazard function of $T$ given $U$. The following lemma is useful throughout our discussion.

\begin{lemma}
\label{lemma::rare-hr}
If $T$ is a rare time-to-event outcome, we have the following approximation:
$$
\lambda(t) \approx \int \lambda(t\mid u) F(du).
$$
\end{lemma}

\begin{proof}
[Proof of Lemma \ref{lemma::rare-hr}.] Similar to the case with discrete $U$ \citep{Vanderweele::2013}, we have $S(t\mid u) \approx 1$ for rare outcome, and therefore
\begin{eqnarray*}
\lambda(t) ={ f(t)\over S(t)} = \frac{ \int \lambda(t\mid u) S(t\mid u) F(du)  }{ \int S(t\mid u)  F(du) } 
\approx  \frac{ \int \lambda(t\mid u)   F(du)  }{ \int  F(du) }  =  \int \lambda(t\mid u) F(du).  \\
\end{eqnarray*}
\end{proof}
 
Lemma \ref{lemma::rare-hr} essentially allows ``Law of Total Probability'' type of calculation for the hazard function with rare outcome.

In order to introduce the new bounding factor for hazard ratio, we need more formal notation. Define the potential outcomes for $T$ as $T(1)$ and $T(0)$ with hazard functions $\lambda^{(1)}(t)$ and $\lambda^{(0)}(t)$ and conditional hazard functions can be defined intuitively as $\lambda^{(1)}(t\mid \cdot)$ and $\lambda^{(0)}(t\mid \cdot)$. We define $\lambda_t^*(u) = \lambda(t\mid E=1, U=u)$ and $\lambda_t(u) = \lambda(t\mid E=0, U=u)$ as the conditional hazard functions of $T$ for the exposed and unexposed units within strata $U=u$, respectively. We define $\HR_{UT|E=1} (t)= \max_u \lambda_t^*(u) /  \min_u \lambda_t^*(u)$ as the maximal hazard ratio function of the confounder $U$ on the outcome $T$ for exposed units, $\HR_{UT|E=0} (t)= \max_u \lambda_t(u) /  \min_u \lambda_t(u)$ for unexposed, and their maximum, denoted by $ \HR_{UT}(t) = \max \{\HR_{UT|E=1} (t) , \HR_{UT|E=0} (t) \} $, as the maximal hazard ratio function of the confounder $U$ on the outcome $T$. Note that the hazard ratios are time-dependent.

If the exposure and the outcome are unconfounded given $U$ and the observed covariates $C$ (which is omitted in conditional probablities for simplicity), Lemma \ref{lemma::rare-hr} allows us to write the true causal hazard ratios for the exposed, unexposed, and the whole population as
\begin{eqnarray*}
\SHR_{ET+} (t)&=&   \frac{\lambda^{(1)}(t\mid E=1)}{\lambda^{(0)}(t\mid E=1)}   
\approx \frac{   \int \lambda_t^*(u)  F_1(du)   }{   \int \lambda_t(u)  F_1(du)     } ,\\
\SHR_{ET-} (t)&=&   \frac{\lambda^{(1)}(t\mid E=0)}{\lambda^{(0)}(t\mid E=0)}   
\approx \frac{   \int \lambda_t^*(u)  F_0(du)   }{   \int \lambda_t(u)  F_0(du)     } ,\\
\SHR_{ET} (t)&=&   \frac{\lambda^{(1)}(t)}{\lambda^{(0)}(t )}   
\approx \frac{   \int \lambda_t^*(u)  F(du)   }{   \int \lambda_t(u)  F(du)     } ,
\end{eqnarray*}
and the observed harzard ratio as
$$
\HR_{ET} (t)= \frac{ \lambda(t\mid E=1) }{ \lambda(t\mid E=0)  }
\approx  \frac{   \int \lambda_t^*(u)  F_1(du)   }{   \int \lambda_t(u)  F_0(du)     }.
$$
With categorical $U$ taking values $0,1,\ldots, K-1$, the true causal hazard ratios can be approximated by the following standardized hazard ratios:
\begin{eqnarray*}
\SHR_{ET+} (t)&\approx&  \frac{   \sum_{k=0}^{K-1} \lambda_t^*(k) \pr(U=k\mid E=1)     }{ \sum_{k=0}^{K-1} \lambda_t(k) \pr(U=k\mid E=1) } ,\\
\SHR_{ET-} (t)&\approx&   \frac{   \sum_{k=0}^{K-1} \lambda_t^*(k) \pr(U=k\mid E=0)     }{ \sum_{k=0}^{K-1} \lambda_t(k) \pr(U=k\mid E=0) } ,\\
\SHR_{ET} (t)&\approx&   \frac{   \sum_{k=0}^{K-1} \lambda_t^*(k) \pr(U=k)     }{ \sum_{k=0}^{K-1} \lambda_t(k) \pr(U=k) }
\end{eqnarray*}
The confounding hazard ratios are defined as 
$$
\CHR_{ET+}(t) = \frac{\HR_{ET} (t)}{\SHR_{ET+} (t)},\quad
\CHR_{ET-}(t) = \frac{\HR_{ET} (t)}{\SHR_{ET-} (t)},\quad
\CHR_{ET}(t) = \frac{\HR_{ET} (t)}{\SHR_{ET} (t)}.
$$

Analogous to the results for the relative risk, we have the following propositions for the hazard ratio. The proofs are straightforward if we replace $\{ r(\cdot), r^*(\cdot) \}$ in the proofs for the relative risk by $\{  \lambda_t(\cdot), \lambda_t^*(u) \}.$

\begin{proposition}
\label{prop::weighted-hr}
For rare time-to-event outcome, we approximately have
\begin{eqnarray*}
\SHR_{ET}(t)  &=& w_t \SHR_{ET+}(t) + (1-w_t) \SHR_{ET-}, \\ 
1/\CHR_{ET}(t) &=& w_t/ \CHR_{ET+}(t) + (1-w_t)/ \CHR_{ET-}(t)  ,
\end{eqnarray*}
where $w_t$ is a weight between zero and one:
$$
w_t = \frac{ f   \int   \lambda_t(u) F_1(du)    }{  f  \int   \lambda_t(u) F_1(du)  + (1-f)  \int   \lambda_t(u) F_0(du) } \in [0,1] .
$$
\end{proposition}

Define the time-varying bounding factor as 
$$
\BF_U(t) = { \GRR_{EU}\times \HR_{UT}(t) \over   \GRR_{EU} + \HR_{UT}(t)  - 1},
$$
which is also time-dependent. The confounding hazard ratios can be bounded by the bounding factor, as shown in the following proposition.

\begin{proposition}
\label{prop::chr-bound}
For rare time-to-event outcome, we approximately have
$$
\CHR_{ET+}(t) \leq  \BF_U(t) ,\quad
\CHR_{ET-}(t) \leq  \BF_U(t) ,\quad
\CHR_{ET}(t) \leq  \BF_U(t) .
$$
\end{proposition}

\begin{proposition}
\label{prop::cornfield-hr}
The implied Cornfield conditions for the hazard ratio from Proposition \ref{prop::chr-bound} are
\begin{eqnarray*}
\GRR_{EU} &\geq & \max_t \CHR_{ET}(t),\\
\HR_{UT}(t) &\geq & \CHR_{ET}(t),\\
\max\{ \GRR_{EU},   \HR_{UT}(t)  \} &\geq &   \CHR_{ET}(t) + \sqrt{\CHR_{ET}(t) \{   \CHR_{ET}(t) - 1\}   } .
\end{eqnarray*}
\end{proposition}

If a proportional hazards model \citep{Cox::1972} for the outcome is used as is often the case in practice, all the above exposure-outcome hazard ratio reduce to a constant $\HR_{ET}(t) = \HR_{ET}$.
The above discussion works well for an exposure that is apparently causative at time $t$ on the harzard ratio scale. If at some time point $t$, the exposure is apparently preventive, then the above discussion needs to be modified. To be more specific, we need to modify the definition of $\GRR_{EU}$ as in Section \ref{sec::preventive-exposure}, and the confounding hazard ratios above are replaced by their reciprocals. 
Likewise similar results on the hazard difference scale hold as those on the risk difference scale in eAppendix A.\ref{sec::rd-rr} provided that the outcome is relatively rare.

\section{A Bounding Factor for General Nonnegative Outcomes}

The discussion above assumes a binary outcome $D$, and in fact all the proofs only use the property that $r(u)$ and $r^*(u)$ are nonnegative. Therefore, the bounding factor also applies to any nonnegative outcomes (counts, continuous positive outcome, etc), if we modify the definitions of $r(u), r^*(u)$, and $\RR_{UD}$ in the following way. For general nonnegative outcomes, we define $r^*(u) = \E(D\mid E=1, U=u) $ and $r(u) = \E(D\mid E=0, U=u)$ as the expectations of the outcome within stratum $U=u$ with and without exposure. Define $\MR_{UD|E=1} = \max_u r^*(u) / \min_u r^*(u)$ and $\MR_{UD|E=0} = \max_u r(u) / \min_u r(u)$ as the mean ratios of $U$ on $D$ with and without exposure, and $\MR_{UD} = \max( \MR_{UD|E=1}, \MR_{UD=0})$ as the maximum of these two mean ratios. Note that when $D$ is binary, $r(u)$ and $r^*(u)$ reduce to probabilities, and the mean ratios reduce to the relative risks. 

The observed mean ratio of the exposure on the outcome is
$$
\MR_{ED} = \frac{   \int \E(D\mid E=1, U=u) F_1(du)    }{  \int \E(D\mid E=0, U=u)  F_0(du) } = \frac{  \int r^*(u)  F_1(du)  }{  \int r(u) F_0(du) } . 
$$
The true causal mean ratio of the exposure on the outcome for exposed is
$$
\SMR_{ED+} =    \frac{   \int \E(D\mid E=1, U=u) F_1(du)    }{  \int \E(D\mid E=0, U=u)  F_1(du) } = \frac{  \int r^*(u)  F_1(du)  }{  \int r(u) F_1(du) },
$$
the true causal mean ratio of the exposure on the outcome for unexposed is
$$
\SMR_{ED+} =    \frac{   \int \E(D\mid E=1, U=u) F_0(du)    }{  \int \E(D\mid E=0, U=u)  F_0(du) } = \frac{  \int r^*(u)  F_0(du)  }{  \int r(u) F_0(du) },
$$
and the true causal mean ratio of the exposure on the outcome for the whole population is
$$
\SMR_{ED+} =    \frac{   \int \E(D\mid E=1, U=u) F(du)    }{  \int \E(D\mid E=0, U=u)  F(du) } = \frac{  \int r^*(u)  F(du)  }{  \int r(u) F(du) }.
$$

Define the bounding factor as
$$
\BF_U =  { \GRR_{EU}\times \MR_{UD} \over  \GRR_{EU} +  \MR_{UD}  - 1 }.
$$
Since the discussion in Section \ref{sec::cornfield-alternative} still holds, the proofs for the following propositions are the same as those in Appendices A.2 and A.4.
First, we have the following bounding factor for nonnegative outcomes:
\begin{proposition}
\label{prop::nonnegative}
$$
\CMR_{ED+} = \frac{\MR_{ED}}{\SMR_{ED+}} \leq   \BF_U, \quad 
\CMR_{ED-} = \frac{\MR_{ED}}{\SMR_{ED-}} \leq \BF_U,\quad 
\CMR_{ED} = \frac{\MR_{ED}}{\SMR_{ED}} \leq  \BF_U.
$$
\end{proposition}

In practice, we might also be interested in the average causal effect of the exposure on the outcome on the difference scale. The observed mean difference 
of the exposure on the outcome is
$$
   \E(D\mid E=1) - \E(D\mid E=0)  \equiv m_1 - m_0 .
$$
The average causal effect of the exposure on the outcome for exposed is
$$
\ACE_{ED+} =     \int \E(D\mid E=1, U=u) F_1(du)    - \int \E(D\mid E=0, U=u)  F_1(du)    = m_1 -  \int   r(u) F_1(du),
$$
the average causal effect of the exposure on the outcome for unexposed is
$$
\ACE_{ED+} =     \int \E(D\mid E=1, U=u) F_0(du)    - \int \E(D\mid E=0, U=u)  F_0(du)  =   \int  r^*(u)   F_0(du) - m_0 ,
$$
and the average causal effect of the exposure on the outcome for the whole population is
\begin{eqnarray*}
\ACE_{ED} &=&          \int \E(D\mid E=1, U=u) F(du)    - \int \E(D\mid E=0, U=u)  F(du)\\  
&=&f \ACE_{ED+}  + (1-f) \ACE_{ED-}.
\end{eqnarray*}

Similar to the discussion in Section \ref{sec::rd-rr} for the risk difference with sensitivity parameters expressed on the risk ratio scale, we have the following proposition about the average causal effect.

\begin{proposition}
\label{prop::ace-rr}
For nonnegative outcomes, the lower bounds for the average causal effects are
\begin{eqnarray*}
&&\ACE_{ED+}  \geq   m_1 - m_0\times \BF_U,\\ 
&&\ACE_{ED-} \geq  m_1/\BF_U - m_0, \\
&&\ACE_{ED} \geq   (     m_1    - m_0\times \BF_U )  \times  \left\{    f  + (1-f)/\BF_U  \right\}
=(    m_1 / \BF_U   - m_0)  \times  \left\{    f\times \BF_U + (1-f)  \right\} .
\end{eqnarray*}
\end{proposition}

We can also obtain similar forms of the conclusion for apparently preventive exposure, for average causal effects averaged over observed covariates, and for corresponding Cornfield conditions. The only difference is that $(p_1, p_0)$ is replaced by $(m_1, m_0)$.

\bibliographystyle{unsrt}
\bibliography{cornfield}

\end{document}